\documentclass[a4paper,11pt]{amsart}
\usepackage[latin1]{inputenc}
\usepackage{amsfonts,amssymb,amsmath,latexsym,verbatim,amscd}
\usepackage{amsthm,amsmath,amssymb,amscd}
\usepackage{amssymb}
\newcommand{\mnote}[1]{\marginpar{%\color{red}%
\raggedright\tiny $\leftarrow$ #1}}
\newcounter{mnotecount}[section]

\newcommand{\rmnote}[1]{}%{\mnote{#1}}
\newcommand{\erw}[1]{\mnote{{\bf erw:} #1}}

\newcommand{\di}{{d}}
\newcommand{\dvol}[1]{{d}vol_{#1}}

\newcommand{\nor}[2]{\|{#1}\|_{#2}}
\newcommand{\al}{\alpha}
\newcommand{\be}{\beta}
\newcommand{\bigo}[1]{\mathcal{O} \left( #1 \right)}

\newcommand{\zi}{z^{i}}

\newcommand{\ep}{\varepsilon}
\newcommand{\ex}{{\rm{e}}}
\newcommand{\te}{\theta}

\newcommand{\gep}{g_{\ep}}

\newcommand{\RK}{{\mathbb{R}^k}}
\newcommand{\EN}{{\mathbb{S}^n}}
\newcommand{\RN}{{\mathbb{R}^n}}
\newcommand{\HN}{{\mathbb{H}^n}}

\newcommand{\dv}[1]{{\rm div}_{#1}}
\newcommand{\tr}[1]{{\rm tr}_{#1}}
\newcommand{\DI}[1]{ D_{#1}}

\newtheorem{teor}{Theorem}

\newtheorem{prop}[teor]{Proposition}
\newtheorem{Remark}[teor]{\sc Remark}
\newtheorem{Definition}{Definition}

\date{March 22, 2010}%{\today}

\begin{document}

\title[Refined gluing for Einstein constraint equations]
{Refined gluing for Vacuum Einstein constraint equations}

\author[E. Delay]{Erwann Delay}
\address{Erwann Delay,
Laboratoire d'analyse non lin\'eaire et g\'eom\'etrie, Facult\'e
des Sciences, 33 rue Louis Pasteur, 84000 Avignon, France}
\email{Erwann.Delay@univ-avignon.fr}
\urladdr{http://www.math.univ-avignon.fr/Delay}
\author[L. Mazzieri]{Lorenzo Mazzieri}
\address{Lorenzo Mazzieri, SISSA - International School for Advanced Studies, via Beirut 2-4, I-34014 Trieste, Italy}
\email{mazzieri@sissa.it}

\maketitle

\begin{abstract}
\noindent We first show that the connected sum along submanifolds introduced by the second author for compact
initial data sets of the vacuum Einstein system
%\cite{Mazz 3} (and \cite{Mazz}, \cite{Mazz 2})
%for compact manifolds
can be adapted to the asymptotically Euclidean and to the asymptotically hyperbolic context. Then, we prove that in any case, and generically, the gluing procedure can be localized, in order to obtain new solutions which coincide with the original ones outside of a neighborhood of the gluing locus.
\end{abstract}

\tableofcontents

\section{Introduction}
\label{s:intro}

It is well known \cite{Ch-B} that a vacuum solution $(Z,\gamma)$ for the Einstein  system ${Ric}_\gamma =0$, where $Z$ is an $(m+1)$-dimensional manifold and $\gamma$ is a Lorentzian metric, may be obtained from solutions to the vacuum Einstein constraint equations on an $m$-dimensional space-like Riemannian submanifold $(M,\tilde g)$ of $Z$ (for further details see eg. \cite{Bar-Ise}). To fix the notations, we say that the triple $(M,\tilde g, \tilde \Pi)$, where $M$ is a smooth manifold, $\tilde g$ is a Riemannian metric and $\tilde \Pi$ is a symmetric $(2,0)$ tensors is a solution to the vacuum Einstein constraints equations if the following relationships are satisfied
\begin{eqnarray}
\label{ECE 1}
J(\tilde g, \tilde \Pi):=\dv{\tilde g} \, \tilde \Pi - \di \big( \tr{\tilde g} \, \tilde\Pi \big) & = & 0  \\
\label{ECE 2} \rho(\tilde g,\tilde\Pi):=R_{\tilde g} - |\tilde\Pi|_{\tilde g}^2 + \big( \tr{\tilde g} \, \tilde\Pi \big)^2 & =
& 0 \, .
\end{eqnarray}
Here $\dv{\tilde g}$ and $\tr{\tilde g}$ are respectively the divergence operator and the trace operator computed with respect to the metric $\tilde g$ and $R_{\tilde g}$ is the scalar curvature of the metric $\tilde g$. In the following we will also refer to the triple $(M, \tilde g, \tilde\Pi)$ as (vacuum) initial data set or Cauchy data set. We also point out that when the evolution $(Z,\gamma)$ of the initial data set is considered according to the hyperbolic formulation of the vacuum Einstein system, then $\tilde g$ and $\tilde\Pi$ turn out to be respectively the Riemannian metric and the second fundamental form induced by the Lorentzian metric $\gamma$ on the space-like slice $M$.

%
%In fact the solutions to the constraints
%form a suitable set of vacuum initial data for the hyperbolic Cauchy
%problem (for further details see \cite{Bar-Ise}). More precisely,
%when we are talking about a solution of the constraints we refer to
%a triple $(M,g,\Pi)$, where $M$ is a smooth manifold and $g$ and
%$\Pi$ are symmetric $(2,0)$ tensors (respectively the induced
%Riemannian metric and the second fundamental form), verifying the
%relationships
%\begin{eqnarray}
%\label{ECE 1}
%J(g,\Pi):=\dv{g} \, \Pi - \di \left( \tr{g} \, \Pi \right) & = & 0  \\
%\label{ECE 2} \rho(g,\Pi):=R_g - |\Pi|_{g}^2 + \left( \tr{g} \, \Pi \right)^2 & =
%& 0 \, ,
%\end{eqnarray}
%where $\dv{g}$ and $\tr{g}$ are respectively the divergence operator
%and the trace operator computed with respect to the metric $g$ and
%$R_g$ is the scalar curvature of the metric $g$.

Of interest in this paper are constant mean curvature (briefly CMC) initial data sets, this means that $\tau:= \tr{\tilde g} \, \tilde\Pi$ is a constant. In the case the system above becomes equivalent to a semi-decoupled system. In fact, following \cite{Ch-B}, \cite{IMaP} and \cite{IMP}, one can split the second fundamental form $\tilde\Pi$ into trace free and pure trace parts
\begin{eqnarray}
\tilde\Pi & = & \tilde\mu \, + \, \frac{\tau}{m} \,\, \tilde g \, ,
\end{eqnarray}
where $\tilde\mu$ is a symmetric $2$-tensor such that $\tr{\tilde g} \, \tilde\mu = 0$. Applying the so called conformal method, we set
\begin{eqnarray}
\label{conf. change 1}
\tilde g \,\,\, =: \,\,\,  u^{{4}/{m-2}} \, {g} \quad & \quad \hbox{and} \quad  & \quad
%\label{conf. change 2}
\tilde \mu \,\,\, =: \,\,\, u^{-2} \, {\mu} \, ,
\end{eqnarray}
where the conformal factor $u$ is a positive smooth function on $M$. It is now straightforward to check that $\tilde g$ and $\tilde\Pi$ verify the Einstein constraint equations (\ref{ECE 1}) and (\ref{ECE 2}) if
and only if ${g}, {\mu}$ and $u$ satisfy
\begin{eqnarray}
\label{TT}
\begin{cases}
\,\,\tr{{g}} \, {\mu} & = \,\,\, 0 \\
\,\,\dv{{g}} \, {\mu} & = \,\,\, 0
\end{cases}
%\label{traceless}
%\tr{\bar{g}} \, \bar{\mu} & = & 0 \\
%\label{transverse}
%\dv{\bar{g}} \, \bar{\mu} & = & 0 \\
\end{eqnarray}
%\begin{eqnarray}
%\label{Lic} {\rm Lic}_{\bar{g}} (u) & = & 0 \, ,
%\end{eqnarray}
%where ${\rm Lic}$ is the semilinear elliptic operator given by
\begin{eqnarray}
\label{Lic}
\Delta_{{g}} \, u \, - \, c_m \,
R_{{g}} \,\, u \, + \, c_m \, | {\mu} \,|_{{g}}^2 \,\,
u^{- \frac{3m-2}{m-2}} \, - \, c_m \, \tfrac{m-1}{m}  \, \tau^2 \,\,
u^{\frac{m+2}{m-2}} & = & 0 \,,
\end{eqnarray}
with $c_m =  \tfrac{m-2}{4\,(m-1)}$. Notice that our Laplacian is negative definite.

Therefore, starting with a metric ${g}$ and a real number $\tau$, one can obtain a $\tau$-CMC solution to the Einstein constraints by producing a symmetric $(2,0)$-tensor $\mu$ which verifies \eqref{TT} (for short a TT-tensor) together with a smooth positive solution $u$ to (\ref{Lic}), which will play the role of the conformal factor. Using then \eqref{conf. change 1} it is easy to recover the triple $(M, \tilde g, \tilde\Pi)$.

In this context and due to the physical meaning, the equation (\ref{ECE 1}) (or equivalently the second equation in (\ref{TT})) is known as the momentum constraint, whereas the equation (\ref{ECE 2}) is the so called Hamiltonian constraint and modulo the conformal transformations it corresponds to the Lichnerowicz equation (\ref{Lic}).

\subsection{Conformal gluing}

In the spirit of \cite{Mazz 3}, we suppose now to start with two Cauchy data sets in the semi-decoupled formulation of the constraints, namely two solutions $(M_i  , \, {g}_i  , \, {\mu}_i  , \, u_i, \, \tau)$, $i=1,2$ to equations (\ref{TT}) and (\ref{Lic}) with the same constant mean curvature $\tau$
%(notice that via the conformal changes $\tilde{g}_i = {u}_i^{4/(m-2)} \, {g_i}$ and $\tilde{\mu}_i = u_i^{-2} \, \mu_i$ this corresponds to considering two sets of $\tau$-CMC solutions $(M_i , \, \tilde{g}_i  , \, \tilde{\Pi}_i = \tilde{\mu_i} +  (\tau / m)
%\tilde{g_i})$, $i=1,2$ to equations (\ref{ECE 1}), and (\ref{ECE 2}))
and we construct the generalized connected sum of
the $m$-dimensional manifolds $M_1$ and $M_2$ along a common isometrically embedded $k$-dimensional Riemannian submanifold $(K, g_K)$ with codimension $n := m-k \geq 3$. This construction consists in excising a small $\ep$-tubular
neighborhood (i.e. a tubular neighborhood of size $\ep \in (0,1)$)
of $K$ in both the starting manifolds and in identifying the
differentiable structures along the leftover boundaries as explained
in \cite{Mazz} and summarized in Section 2. The reasons for requiring $n\geq 3$ is discussed in \cite{Mazz} and \cite{Mazz 2} as well as in \cite{Mazz 3}. Our goal is then to endow the new manifold $M_\ep \, = \, M_1 \, \sharp_{K,\ep} \, M_2$ with a Riemannian structure $g_{\ep}$ and a symmetric TT-tensor $\mu_\ep$ such that it is possible to find a solution $u_\ep$ to the Lichnerowicz equation $(\ref{Lic})$.

In \cite{Mazz 3} the second author treated the case where $M_1$, $M_2$ and $K$ are compact. Here we extend the conformal gluing construction to the case where $M_1$ and $M_2$ are asymptotically Euclidean (AE) and asymptotically hyperbolic (AH) and the submanifold $K$ is still compact. The generalized connected sum along a non compact submanifold embedded in asymptotically Euclidean or asymptotically hyperbolic initial data sets seems to require some extra assumption on the behavior of $K$ at infinity and will be the object of further investigations.

As in \cite{Mazz 3}, we will obtain a control of the new solutions in terms of the original ones. In fact, setting
\begin{eqnarray*}
\tilde{g}_\ep \,\,\, := \,\,\, u_\ep^{\frac{4}{m-2}} \, g_\ep \,  & \quad \hbox{and}  & \quad
\tilde{\Pi}_\ep \,\,\, := \,\,\, u_\ep^{-2}\, \mu_\ep \, + \, \frac{\tau}{m}
\,\, u_\ep^{\frac{4}{m-2}} \, g_\ep \, , \\
\tilde{g}_i \,\,\, := \,\,\, u_i^{\frac{4}{m-2}} \, g_i \,  & \quad \hbox{and}  & \quad
\tilde{\Pi}_i \,\,\, := \,\,\, u_i^{-2}\, \mu_i \, + \, \frac{\tau}{m}
\,\, u_i^{\frac{4}{m-2}} \, g_i \, ,
\end{eqnarray*}
we will have that $\tilde{\gep} \rightarrow \tilde{g}_i$ and $\tilde{\Pi}_\ep \rightarrow \tilde{\Pi}_i$ on the compact sets of $M_i \, \setminus \,K$, with
respect to the $\mathcal{C}^r$ topology, for every $r \in \mathbb{N}$, $i=1,2$. This fact turns out to be very important when we will discuss the localization procedure.

The strategy of the proof is the same as in the compact case \cite{Mazz 3} and here we just outline it briefly. The first step is the construction of a family of approximate solution metrics $g_\ep$, $\ep \in (0,1)$, which coincide with $g_i$, $i=1,2$ away from $K$. In correspondence to each $\ep$ we will then produce a $g_\ep$-trace free symmetric 2-tensor $\mu$ by means of cut-off functions. We notice that $\mu$ coincide with $\mu_i$, $i=1,2$, in the region where $\gep$ coincide with the original metrics.

The next step is the resolution of the momentum constraint equation, which is the linear part of the semi-decoupled system. This will be done by adding to $\mu$ a correction term of the form $D_{\gep} X$, where $D_{\gep}$ denotes the conformal Killing operator acting on vector fields
\begin{eqnarray}
\label{deformation operator}
\DI{\gep} \, X  & := & \frac{1}{2} \,\,
\mathcal{L}_X \, \gep \, - \, \frac{1}{m} \, \left( \dv{\gep} \, X \right)
\,\cdot \, \gep \, .
\end{eqnarray}
%We also recall that it is the negative of the formal adjoint of
%the divergence applied to symmetric trace free $2$-tensors. More
%precisely
%\begin{eqnarray}
%\label{deformation=-div*}  \DI{g} & = & - \, \left( \sharp \, \dv{g}
%\right)^* \, .
%\end{eqnarray}
As a consequence, the momentum equation becomes
\begin{eqnarray}
\label{vector laplacian}
L_{\gep} \, X & = & \sharp \, \dv{\gep} \, \mu \, ,
\end{eqnarray}
where $L_{\gep} :=  \sharp \, \dv{\gep} \, \circ
\, \DI{\gep}$ is a second order elliptic partial differential operator known as vector Laplacian. The fact that $L_{\gep}$ is self-adjoint follows from the algebraic identity $\DI{\gep}  =  - \, \left( \sharp \, \dv{\gep}\right)^* $.

A crucial step in the study of equation \eqref{vector laplacian} is to provide the solution with an {\em a priori} estimate independent of $\ep$. This will allows us to control the squared pointwise norm $|\mu_\ep|^2$ of the TT-tensor $\mu_\ep := \mu + \DI{\gep}X$ which appears in the Lichnerowicz equation. To obtain this $\ep$-uniform {\em a priori} bound we need to introduce the following notion of nondegeneracy
\begin{Definition}[Nondegeneracy]
A compact Riemannian manifold $(M,g)$ is said to be {\em nondegenerate} if there are no nontrivial conformal Killing vector fields on it or equivalently if the conformal Killing operator $\DI{g}$ is injective.
\end{Definition}
%\smallskip
%\noindent \textbf{Non-degeneracy condition:} \textit{There are no
%nontrivial conformal Killing vector fields on either $(M_1,g_1)$ or $(M_2,g_2)$.} Equivalently both $\DI{g_1}$ and $\DI{g_2}$ are injective.
%\smallskip
We will see that the injectivity of the conformal Killing operator is automatically achieved in the asymptotically Euclidean (AE) and in the asymptotically hyperbolic (AH) situation, up to a careful choice of the functional setting. Notice that due the different geometric construction this assumption turns out to be slightly different from the nondegeneracy condition required in \cite{IMP}. In fact the IMP gluing works under the assumption that there are no nontrivial conformal Killing vector fields on $(M_1, g_1)$ and $(M_2, g_2)$ which vanish at the excised points.

The last step of the overall strategy is the perturbative resolution of the Lichnerowicz equation. In analogy with \cite{IMP} and \cite{Mazz 3} this will be done by studying the linearized operator about $\gep$.
Here as well the key point is the obtention of $\ep$-uniform {\em a priori} bounds. With these at hand, one can then carry out a fixed point argument which yields a solution to the Lichnerowicz equation. Again, the proof of the uniform {\em a priori} bounds requires a sort of non-degeneracy assumption, more precisely we need to suppose that both $\tilde\Pi_1$ and $\tilde\Pi_2$ are non identically zero. This is sufficient to guarantee that the linearized Lichnerowicz operators about the original metrics
\begin{eqnarray}
\label{injective} \qquad \qquad \qquad \Delta_{g_i} \, - \,
|{\mu}_i|^2_{{g_i}} \, - \, \tau^2/m \, , & &
i=1,2\,
\end{eqnarray}
are injective. We remark that, also in this case, for asymptotically Euclidean and asymptotically hyperbolic summands the injectivity of the linearized Lichnerowicz operator will follows directly from a judicious choice of the functional setting, without any further assumption.

Following the strategy summarized above, we can prove the following generalized conformal gluing result
\begin{teor}[Conformal gluing]
\label{teo 1} Let $(M_1,\, \tilde{g}_1, \, \tilde{\Pi}_1)$ and
$(M_2,\, \tilde{g}_2, \, \tilde{\Pi}_2)$ be two $m$-dimensional CMC solutions to the Einstein constraint equations (\ref{ECE 1})-(\ref{ECE 2}) having the same constant mean curvature $\tau$ and such that $M_i$ is either compact or AE or AH. If for some $i=1,2$ $M_i$ is compact, we further assume that it is {\em nondegenerate} and that $\tilde\Pi_i$ is non identically zero. Let $(K,\, \tilde{g}_K)$ be a common isometrically embedded
$k$-dimensional compact sub-manifold of codimension $n\, := \, m-k \, \geq
\, 3$ such that the normal bundles of $K$ in $M_1$ and in $M_2$ are
diffeomorphic. Then
\vspace{0.2cm}
\begin{enumerate}
\item There exists a real number $\ep_0 \in (0,1)$ such
that for every $\ep \in (0, \ep_0)$ it is possible to endow the
$\ep$-generalized connected sum $M_\ep \, = \, M_1 \,
\sharp_{K,\,\ep} \, M_2$ of $M_1$ and $M_2$ along $K$ with a metric
$\tilde{g}_\ep$ and a second fundamental form $\tilde{\Pi}_\ep$ such
that the triple $(M_\ep, \, \tilde{g}_\ep , \, \tilde{\Pi}_\ep )$ is
still a $\tau$-CMC solution to the Einstein constraint equations.

\vspace{0.2cm}

\item The new metric $\tilde{g}_\ep$ and the new symmetric TT-tensor $\tilde{\mu}_\ep$ tend to the original metric
$\tilde{g}_i$ and to the original symmetric TT-tensor
$\tilde{\mu}_i$ with respect to the $\mathcal{C}^r$-topology on the compact subsets of $M_i \, \setminus \, K $, for every $r \in \mathbb{N}$ and $i=1,2$, when the geometric parameter $\ep$ tends to zero.
\end{enumerate}
\end{teor}

%%%%%%%%%%%%%%%%%%%%%%%%%%%%%%%%%%\`u

\subsection{Localized gluing}
Once the conformal gluing is performed, we refine the result of Theorem \ref{teo 1} by showing that one can actually produce new solutions to the Einstein constraints on $M_\ep$ which coincide exactly with the starting ones outside a suitable neighborhood of the gluing locus. This will be done by re-gluing the old solutions to the conformally glued ones far enough from the so called poly-neck region.
%
%
% we show that, in a second hand, we can reglue the new metric with the metrics
%$g_i$ we started with, outside  neighborhood of the neck.

To present our second issue, we recall that, in analogy with \cite{Mazz 3}, the conformal gluing occurs, for $\ep\leq \ep_0<1$, in the regions $K_1^i\backslash \overline K^i_\ep$, where for an arbitrary $r>0$ we have set
$$
K_r^i \,\, := \,\, \{ \, m\in M_i \,\,\, : \,\,\, m=\exp^{g_i}_z(x),\; z\in K, \; x\in (T_zK)^\bot, |x|_{g_i}<r \, \} \, .
$$
Furthermore, we recall that $( \tilde{g}_\ep , \, \tilde{\Pi}_\ep )$ are conformal deformations of
$(\tilde{g_i}, \, \tilde{\Pi_i})$ on
$M_i\setminus K_1^i$, $i=1,2$. Having these observations in mind, we will show that in \emph{generic} situations the conformal gluing can be localized. Our second result reads
%performed so that the new initial data coincide  with the original
%ones away from a small neighborhood of the $K^i$'s:
\begin{teor}[Localized gluing]
\label{Tloc} Let $\ep_0<1/2$ be small enough so that the  boundary's
 $\partial K_\ep^i$ are smooth manifolds for $\ep\le 2\ep_0$. Suppose
that there exists $0<\ep_1\leq \ep_0$ such that the set of KIDs on $\Gamma^i (\ep_1,2\ep_1):= K_{2\epsilon_1}^i\setminus \overline{K_{\epsilon_1}^i}$ is trivial. Then, there exists
 $\ep_2\leq \ep_1$ and a family of smooth vacuum initial data sets $( M_\ep,  \hat g_\ep,
 \hat\Pi_\ep)$, $\ep\le \ep_2$ such that
 $$(\hat g_\ep, \hat \Pi_\ep)= ( \tilde{g}_i,\tilde\Pi_i) \ \mbox{ on } \  M_i\setminus
K_{2\ep_1}^i \,\, \quad \,\, i=1,2
$$
and
$
(\hat g_\ep, \hat \Pi_\ep)$ coincide with  $(\tilde g_\ep, \tilde \Pi_\ep)$ in the
poly-neck region.
 \end{teor}
Note that {\em a priori} these solutions fail to be CMC in the region $\Gamma^i (\ep_1,2\ep_1)$, $i=1,2$.

To explain better the meaning of our hypothesis we also recall that the KIDs are by definition the elements in the kernel of the $L^2$ formal adjoint $D\Phi^*(g,\Pi)$ of the linearization at $(g,\Pi)$ of the constraint operator $\Phi(g,\Pi) \,:= \,(\rho(g,\Pi),J(g,\Pi))$. According to \cite{BCS}, one has that for a {\em generic} metric the set of KID's on $\Gamma(\ep_1,2\ep_1)$ is trivial for all $\ep_1$.

Finally, we observe that in the case where the set of KIDs is trivial only on say $\Gamma^1 (\ep_1,2\ep_1)$, then the same construction gives initial data sets which coincide with $(\tilde g_1,\tilde \Pi_1)$ on $M_1\setminus K^1_{2\ep_1}$
and with $(\tilde g_\ep, \tilde \Pi_\ep)$ on $M_\ep\setminus (M_1\setminus K^1_{\ep_1})$.

\section{The geometric construction}
\label{geom construction}

In this section we briefly present the construction of the approximate solution metrics $\gep$, $\ep \in (0,1)$ and the consequent construction of the $g_\ep$-trace free 2-tensors $\mu = \mu(\ep)$ via cut-off functions. Even though the construction is local and then it is identical to the one used in \cite{Mazz 3}, we prefer to recall it quickly, in order to fix the notations that will be used throughout the paper.

%

%

%

%The aim of this section is to give a precise description of the
%generalized connected sum and to present a way to construct a family
%of approximate solution metrics $(\gep)_{\ep\in (0,1)}$; these are
%metrics which, when $\ep$ varies in a sufficiently small range, can
%be perturbed to the final metric $\tilde{g}_\ep$ by means of a small
%(i.e. close to one) conformal factor $u_\ep$, this one being a
%solution to the Lichnerowicz equation with respect to the metric
%$\gep$, the constant mean curvature $\tau$ and a suitable TT-tensor
%$\mu_\ep$. At the end of this section we also present the
%construction of a symmetric $\gep$-trace free tensor $\mu =
%\mu(\ep)$ by means of a warped cut-off method (then, repairing this
%$\mu$ by means of a suitable symmetric tensor $\sigma_\ep$, we will
%find the TT-tensor $\mu_\ep:= \mu + \sigma_{\ep}$ mentioned above).

Let $(K,g_K)$ be a $k$-dimensional compact Riemannian manifold isometrically embedded in both the $m$-dimensional Riemannian
manifolds $(M_1,g_1)$ and $(M_2,g_2)$, we label the embedding maps
as follows
\begin{eqnarray*}
\iota_i : K \hookrightarrow M_i \,\,, \quad i=1,2.
\end{eqnarray*}
\noindent We assume that the map $\iota_1^{-1}\circ\,\, \iota_2
: \iota_1(K) \rightarrow \iota_2(K)$ is an isometry which extends to a diffeomorphism
between the normal bundles of $\iota_i (K)$ in $(M_i,g_i)$, $i=1,2$.
To simplify the notations and the computations, we suppose, without loss of generality, that the injectivity radius of $K$ in both the manifolds is greater than one, so that we are allowed to manipulate the differential and the metric structure of the original solutions in a fixed size one tubular neighborhood of $K$ in both $M_1$ and $M_2$. This fixed tubular neighborhood will be referred in the following as {\em gluing locus}.

For a fixed $\ep \in (0,1)$, we describe the
construction of the generalized connected sum of
$M_1$ and $M_2$ along $K$ and the definition of the metric $g_\ep$
in local coordinates. The fact that this construction yields a
globally defined metric will follow at once.

Let $U^k$ be an open set in $\mathbb{R}_z^k $ an let $B^{n}$ be the $n$-dimensional open ball of radius one in $\mathbb{R}_x^n$. We recall that $n$ is the codimension of $K$ in $M_i$ an verifies $n := m-k \geq 3$. For $i=1,2$, the map
%$F_i : U^k \times B^{n} \rightarrow W_i \subset M_i$
given by
\begin{eqnarray}
F_i : U^k \times B^{n} \rightarrow W_i \subset M_i \,\, ,  & &   F_i(z,x) \, : = \, \exp^{\,g_i}_{(z,0)}(x) \,
\end{eqnarray}
defines local Fermi coordinates near the coordinate patches
$F_i(\cdot,0)\big( U^k \big)\subset \iota_i(K) \subset M_i$. In
these coordinates, the metric $g_i$ can be decomposed as
\begin{eqnarray}
g_{(i)}(z,x) & = & g^{(i)}_{jl} dz^{j} \otimes dz^{l} + g^{(i)}_{\al
\be}dx^{\al} \otimes dx^{\be} + g^{(i)}_{j \al }dz^{j} \otimes
dx^{\al} \, ,
\end{eqnarray}
where the coefficients are such that
\begin{eqnarray*}
g^{(i)}_{\al \be} = \delta_{\al \be}+ \bigo{|x|^2} & \qquad
\mbox{and} \qquad & g^{(i)}_{j \al } = \bigo{|x|} \, .
\end{eqnarray*}
Also notice that, since $K$ is isometrically embedded in $M_1$ and in $M_2$, we have that
$$
g^{(1)}_{ij} (z,0) \,\, = \,\, g^K_{ij}(z) \,\, = \,\, g_{ij}^{(2)} (z,0) \,.
$$
To define the differential and the metric structure of the generalized connected sum, we consider now, in correspondence to each $\ep \in (0,1)$, the hollow domain $U^k \, \times \, A^n \, (\ep^2,1)$, where, for $r,s \in \mathbb{R}^+$, $A^n \, (r,s)$ denotes the $n$-dimensional annulus $\{r < |x| < s \}$. With the notations of the previous chapter we have that
$$
F_i \, (U^k \times A^n(\ep^2,1)) \,\, = \,\, W_i \cap \Gamma^i(\ep^2,1) \,\, , \quad i=1,2\,.
$$
To identify $W_1 \cap \Gamma^1(\ep^2,1)$ {and} $W_2 \cap \Gamma^2(\ep^2,1) $ it is convenient to set
\begin{eqnarray*}
x & = & \ep \,
\ex^{-t} \, \theta  \quad \hbox{in} \,\, F_1^{-1}\big[ W_1 \cap \Gamma^1(\ep^2,1) \big]
%& \hbox{and}&
\\
x & = & \ep \, \ex^{t} \,
\theta \quad \,\,\,\, \hbox{in} \,\,  F_2^{-1}\big[ W_2 \cap \Gamma^2(\ep^2,1) \big] \,,
%\quad \hbox{in} \,\, F_2^{-1}(W_2) \,\, ,
\end{eqnarray*}
with $\ep \in (0,1)$, $\log \ep  < t < -
\log \ep$, $\theta \in S^{n-1}$. Now, it is immediate to check that both $W_1 \cap \Gamma^1(\ep^2,1)$ and $W_2 \cap \Gamma^2(\ep^2,1)$ are diffeomorphic to $U^k \times (\log\ep,-\log\ep)\times S^{n-1}$ and that $(z, t, \theta)$ yield a local coordinates for $M_\ep$, this automatically provides the local definition of the differential structure of $M_\ep$. Using these coordinates, the expressions of
the two metrics $g_1$ and $g_2$ become respectively
\begin{eqnarray}
g_{1}(z,t,\te) & = & g^{(1)}_{ij} dz^i \otimes dz^j \nonumber \\
& + & u_{\ep}^{(1)} \, ^{\frac{4}{n-2}} \left[ \left(dt \otimes dt +
g^{(1)}_{\lambda \mu} d\theta^\lambda \otimes d\theta^\mu \right) +
g^{(1)}_{t \theta}dt \ltimes d\theta
\right]\\
& + & g^{(1)}_{i t} d\zi \otimes dt + g^{(1)}_{i \lambda} d\zi
\otimes d\te^{\lambda} \nonumber
\end{eqnarray}
and
\begin{eqnarray}
g_{2}(z,t,\te) & = & g^{(2)}_{ij} dz^i \otimes dz^j \nonumber \\
& + & u_{\ep}^{(2)} \, ^{\frac{4}{n-2}} \left[ \left(dt \otimes dt +
g^{(2)}_{\lambda \mu} d\theta^\lambda \otimes d\theta^\mu \right) +
g^{(2)}_{t \theta}dt \ltimes d\theta
\right]\\
& + & g^{(2)}_{i t} d\zi \otimes dt + g^{(2)}_{i \lambda} d\zi
\otimes d\te^{\lambda} \, , \nonumber
\end{eqnarray}
where the compact notation $g_{t \theta} \, dt \ltimes d\theta$ indicates the general component of the normal metric tensor, i.e., a linear combination with smooth coefficients of $dt
\otimes dt$, $d\theta^\lambda \otimes d\theta^\mu$ and $dt \otimes d\theta^\lambda$.
%whereas the coefficients $g_{t
%\theta}$ multiplied by $u_{\ep}^{(i)} \, ^{\frac{4}{n-2}}$, $i=1,2$
%represent the correction to the Euclidean metric $u_{\ep}^{(i)} \,
%^{\frac{4}{n-2}} \left[ \left(dt \otimes dt + g^{(i)}_{\lambda \mu}
%d\theta^\lambda \otimes d\theta^\mu \right) \right]$, $i=1,2$ in our
%coordinate system.\\
%\\
The normal conformal factors are given by
\begin{eqnarray*}
u_{\ep}^{(1)} (t) = \ep^{\frac{n-2}{2}} \ex^{-\frac{n-2}{2}t} &
\qquad \mbox{and} \qquad & u_{\ep}^{(2)} (t) = \ep^{\frac{n-2}{2}}
\ex^{\frac{n-2}{2}t}\, ,
\end{eqnarray*}
whereas $g_{\lambda \mu}^{(1)}$ and $g_{\lambda\mu}^{(2)}$ are the coefficients of the round metrics on $S^{n-1}$. Remark that for $j = 1,2$ we have
\[
\begin{array}{rlllllll}
g^{(j)}_{\lambda \mu} & = & \bigo{1} & \qquad \qquad &
g^{(j)}_{t \theta} & =& \bigo{|x|^2}\\[3mm]
%g^{(j)}_{ \lambda \mu}& = &\bigo{|x|^2}\\
g^{(j)}_{i t} & = & \bigo{|x|^2} & \qquad \qquad &
g^{(j)}_{i \lambda} & = & \bigo{|x|^2} \, .
\end{array}
\]
%and
%\begin{eqnarray*}
%u_{\ep}^{(1)} (t) = \ep^{\frac{n-2}{2}} \ex^{-\frac{n-2}{2}t} &
%\qquad \mbox{and} \qquad & u_{\ep}^{(2)} (t) = \ep^{\frac{n-2}{2}}
%\ex^{\frac{n-2}{2}t}\, .
%\end{eqnarray*}

We choose now a cut-off function $\chi : (\log\ep, -\log\ep)
\rightarrow [0,1]$ which is a non increasing smooth function identically equal to $1$ in $(\log\ep, -1]$ and identically equal to $0$ in
$[1,-\log\ep)$. We choose then another cut-off function $\eta :
(\log\ep, -\log\ep) \rightarrow [0,1]$ which is a non increasing smooth function identically equal to $1$ in $(\log\ep, -\log\ep
-1]$ and such that $\lim_{t\rightarrow -\log\ep} \eta = 0$.
Using these two cut-off functions, we can define a new normal
conformal factor $u_{\ep}$ by setting
\begin{eqnarray}
u_{\ep} (t) & : = & \eta(t) \, u_{\ep}^{(1)} (t) \, + \, \eta(-t) \,
u_{\ep}^{(2)}(t) \, .
\end{eqnarray}
The metric $g_{\ep}$ is then locally defined by
\begin{eqnarray}
\label{gep}
g_{\ep}(z,t,\te) & := & \left(\chi g^{(1)}_{ij} + (1-\chi) g^{(2)}_{ij} \right)dz^i \otimes dz^j \nonumber \\
& + & u_{\ep}^{\frac{4}{n-2}} \left[ dt \otimes dt + \left(\chi
g^{(1)}_{\lambda \mu} + (1-\chi) g^{(2)}_{\lambda \mu}
\right)d\theta^\lambda \otimes d\theta^\mu \right. \nonumber \\
& & \qquad \qquad \qquad + \left. \left(\chi g^{(1)}_{t \theta}
+ (1-\chi) g^{(2)}_{t \theta} \right) dt \ltimes d\theta \right]\\
& + & \left( \chi g^{(1)}_{i t} + (1-\chi) g^{(2)}_{i t} \right)
d\zi
\otimes dt \nonumber \\
& + & \left(\chi g^{(1)}_{i \lambda} + (1-\chi) g^{(2)}_{i
\lambda}\right) d\zi \otimes d\te^{\lambda} \, . \nonumber
\end{eqnarray}

Closer inspection of this expression shows that the only objects
that are not {\it a priori} globally defined on the identification
of the hollow tubular neighborhoods (poly-neck) of $\iota_{1} (K)$ in $M_1$
and $\iota_2 (K)$ in $M_2$ are the functions $\chi$ and $u_{\ep}$
(since the cut-off $\eta$ is involved in its definition). However,
observe that both cut-off functions can easily be expressed as
functions of the Riemannian distance to $K$ in the respective
manifolds. Hence, they are globally defined and the metric $g_{\ep}$
- whose definition can be obviously completed by setting $g_{\ep}
\equiv g_1$ and $g_{\ep} \equiv g_2$ out of the poly-neck - is a
Riemannian metric which is globally defined on the manifold $M_\ep$.

We conclude this section with the definition of the
`proto-TT-tensor' $\mu = \mu(\ep)$. This is a $\gep$-trace free symmetric 2-tensor which in the next section will be corrected to a $\gep$-TT-tensor by adding an $\ep$-uniformly bounded term of the form $\DI{\gep} X$, as anticipated in the introduction. To define $\mu$ we describe a warped cut off
procedure on the side of the polyneck coming from $M_1$. The same
manipulation on the other side provides us with the complete
definition of $\mu$. Let then $\mu_1$ be the given $g_1$-trace-free symmetric tensor on $M_1$. In local coordinates this reads
\begin{eqnarray*}
\label{traceless1} g^{(1)}_{ij} \, \mu_1^{ij} \, + \, 2 \, g_{i
\al}^{(1)} \, \mu_1^{i \al} \, + \, g_{\al \be}^{(1)} \mu_1^{\al
\be} & = & 0 \, .
\end{eqnarray*}
\noindent We are looking for a symmetric tensor $\mu$ which is trace
free with respect to the metric $\gep$. To do that we set
\begin{eqnarray}
\mu^{ij} & = & a(t) \, \cdot \, \mu_1^{ij} \nonumber \\
\mu^{i \al } & = &  a(t) \, \cdot \, \mu_1^{i \al} \\
\mu^{\al \be} & = & b(t) \, \cdot \, \mu_1^{\al \be} \, ,\nonumber
\end{eqnarray}
\noindent where $a$ and $b$ are smooth radial cut-off functions
which are equal to one on $M_1$ and which vanish for $t \, > \, c \,
\log \ep \, + \, 1$, where $0\, < \, c \, < \, 1$ is a constant that
will be determined in Section 4. The definitions of $a$ and $b$ are made more precise below. However, we remark that the warped cut-off
still guarantees the symmetry of $\mu$. Taking into account
(\ref{traceless1}) and the definition of the metric $\gep$ on the
region where $\mu$ is not identically zero, the condition $\tr{\gep}
\, \mu \, = \, 0 $ is equivalent to
\begin{eqnarray*}
0 & = & \label{traceless ep} a \, g^{(1)}_{ij} \, \mu_1^{ij} \, + \,
2 \, a \, g_{i \al}^{(1)} \, \mu_1^{i \al} \, + \, b \, \phi^2
\, g_{\al \be}^{(1)} \mu_1^{\al \be}
\,\,\, = \,\,\, \left[b \, \phi^2 \, - \, a \right] \, \cdot \, g_{\al
\be}^{(1)} \mu_1^{\al \be} \, ,
\end{eqnarray*}
where the normal conformal factor $\phi^2$ defined by
\begin{eqnarray*}
\phi^2 & := & \left[ 1 \, + \, \eta(-t)  \cdot  (u^{(2)}_\ep
/u^{(1)}_\ep) \right]^{\frac{4}{n-2}} \, .
\end{eqnarray*}

It is now straightforward to verify that one can always
choose two smooth cut-off functions satisfying the conditions above
and such that $a \, = \,  \phi^2 \, b $. In particular we will
choose $b \, \equiv \, 1$, for $\log \ep \, < \, t \, < \, c \, \log
\ep$. As a consequence, $a \, = \, b \, + \,
\mathcal{O}\big({\ex^{(n-2)t}}\big)$ on this interval and $\mu$
converges to $\mu_1$ on the compact sets of $M_1 \setminus K$ with
respect to the $\mathcal{C}^2$-topology.

Notice that since $\phi$ depends on $\ep$, also $a$ and $b$ depend on $\ep$. Nevertheless, they and their derivatives do admit an
$\ep$-uniform bound. Finally, let us
observe that
%
%, for every $k \geq 0$, $\left|\nabla^k \, \dv{\gep} \,
%\mu \right|_{\gep} \, (t) \, \rightarrow \, \left|\nabla^k \,
%\dv{g_1} \, \mu_1 \right|_{g_1} \,(\log \ep)\, = \, 0$, as $t$ tends
%to $\log \ep$. Moreover
$\left|\dv{\gep} \, \mu \right|_{\gep} $ and
$\left|\nabla \, \dv{\gep} \, \mu \right|_{\gep}$ are $\mathcal{O}
\big({\ep^{n-2}}\big)$ near the boundary of the poly-neck and
$\ep$-uniformly bounded in the interior.

\section{The momentum constraint: existence of solutions}
%\subsection{The momentum constraint}

The aim of this section is to provide existence and uniform {\em a priori} estimates for solutions to the equation
\begin{eqnarray}
\label{momentum}
L_{\gep} \, X & = & \sharp \, \dv{\gep} \, \mu \,\,  \quad \hbox{in} \,\, M_\ep \,,
\end{eqnarray}
where $g_\ep$ and $\mu$ are respectively the approximate solution metric and the approximate symmetric TT-tensor defined in the previous section. According to the strategy explained in the introduction, we will then obtain a family of TT-tensor by setting $\mu_\ep := \mu + \DI{\gep}X$. We recall that
%the local expression of the conformal killing operator is given by
%\begin{eqnarray}
%\left( \DI{\gep} X \right)_{jk} & = & \frac{1}{2} \, \left[\,
%(\nabla_j X)_k \, + \, (\nabla_k X)_j \, \right] \, - \, \frac{1}{m}
%\, (\nabla_l X)^l \, \cdot \,  g^{(\ep)}_{jk} \, ,
%\end{eqnarray}
%\noindent where $\nabla$ is the Levi-Civita connection of the metric
%$\gep$, and the indices have been lowered by means of the metric
%$\gep$, where needed.
%We also recall that
the vector Laplacian - is defined as $L_{\gep} \, := \, (\DI{\gep})^* \, \cdot
\, \DI{\gep} \, = \, - \, \sharp \, \dv{\gep} \, \cdot \,
\DI{\gep}$. As it is easy to verify, $L_{\gep}$ is a linear elliptic
second order partial differential operator with smooth coefficients
and it is formally self-adjoint. We can think of the vector
Laplacian as acting between the spaces of sections of the tangent bundle with H\"{o}lder
regularity
\begin{eqnarray}
\label{action} L_{\gep} \,\, : \,\, \mathcal{C}^{k+2,\al}(M_\ep,
TM_\ep) \,\, \longrightarrow \,\, \mathcal{C}^{k,\al}(M_\ep,TM_\ep)
\, .
\end{eqnarray}
We recall that for a
general tensor field $T$ the $\mathcal{C}^k$ norm of $T$
is defined by
\begin{eqnarray}
\nor{T}{\mathcal{C}^k} & := & \sum_{j=0}^k \,\,  \sup_{M_\ep} \,\,
| \nabla^j T |_{\gep} \, ,
\end{eqnarray}
where $\nabla$ is the Levi-Civita connection of $\gep$, whereas the
H\"{o}lder seminorm of the $k$-th derivative with H\"older exponent $\al \in
(0,1)$ is defined by
\begin{eqnarray}
\label{holder seminorm} \big[ \nabla^k T \big]_{\al} & := &
\sup_{p \neq q} \,\, \frac{\left| \nabla^k T(p) \, - \, \nabla^k
T(q) \right|_{\gep}}{d_{\gep} (p , q)^\al} \, ,
\end{eqnarray}
where the distance $d_{\gep} (p  , q)$ is supposed to be smaller than
the injectivity radius and, with abuse of notation, the term $\nabla^k T (q)$ is interpreted as its parallel transport from $q$ to $p$ along the unique geodesic joining $p$ and $q$ (in order to give
sense to the subtraction which appears in the numerator above). The
definition of the $\mathcal{C}^{k,\al}$-H\"{o}lder norm is then given by
\begin{eqnarray}
\nor{T}{\mathcal{C}^{k,\al}} & := & \nor{T}{\mathcal{C}^{k}} \, + \, \big[ \nabla^k T \big]_{\al} \, .
\end{eqnarray}

As explained in \cite{Mazz 3}, the existence of solutions to \eqref{momentum} in the compact case (i.e., when both the summands $M_1$ and $M_2$ are compact and hence $M_\ep$ itself is compact) is an easy consequence of the Fredholm alternative. In fact, using integration by parts, it is immediate to check that the right hand side $\sharp \, \dv{\gep} \, \mu$ is $L^2$-orthogonal to ${\rm Ker} \, L_{\gep}^*$. The H\"older regularity of the solution follows then from standard elliptic regularity theory.

In the non compact case we need to use a more sophisticated functional setting in order to provide existence. More precisely, we are going to recall the definitions of asymptotically Euclidean (AE) and asymptotically hyperbolic (AH) Riemannian manifold, together with the notion of weighted H\"older spaces (at infinity). In this context it will be possible to quote some well known result about the mapping properties of the vector Laplacian, insuring the existence of solutions to equation \eqref{momentum} also in the non compact case.

\subsection{Asymptotically Euclidean manifold}
An asymptotically Euclidian (AE) Riemannian manifold $(M,g)$ is a manifold $M$
with a finite number of ends, each of them being diffeomorphic to the exterior
of a ball in $\RN$, and  such that, up to this diffeomorphism, the components of the  metric on each end verifies
%($r$ being
%the standard radius on $\RN$ and $\nu>0$)
$$
|g_{ij}-\delta_{ij}|\leq C r^{-\nu} \, ,
$$
along with appropriate decay of the derivatives. Here $\nu>0$ and $r$ is a function which is equal to the standard radius of $\mathbb{R}^n$ on each end and which is equal to the constant 1 on the remaining compact region of $M$.
%We define now the weighted H\"older spaces of sections of the tangent bundle $C^{\,k,\al}_{AE(-\nu)}(M,TM)$

The weighted H\"older spaces $r^{-\nu}C^{k,\alpha}_{AE}$ are given by the tensor fields such that the norm
\begin{eqnarray*}
\nor{\, \cdot \, }{C^{k,\al}_{-\nu}} & := & \nor{r^\nu \,  \cdot \,}{C^{k,\al}}
\end{eqnarray*}
is well definite and finite.
%Of interest in the treatment of the momentum constraint are the spaces $r^{-\nu} C_{AE}^{k,\al}(M,TM)$
Roughly speaking a function (or a tensor field) $u$ is in $r^{-\nu}C^{k}_{AE}$ if $u\in C^{k}_{loc}$ and on each end $\nabla^{l}u=O(r^{-\nu-l})$, for all $l\leq k$. For further references on this topic, we point out that the spaces defined here coincide with the spaces denoted by $\Lambda^{k,\al}_{AE}$ in \cite{IMP}.

Since the model we want to reproduce is the `horizontal' Euclidean slice in Minkowski space, we will
assume that $\tau=0$, when at least one of the two summands in $M_\ep$ is AE. Because of the nature of the Hamiltonian constraint it is also natural to assume that
$$
\mu\in r^{-\nu/2-1}C^{k,\alpha}_{AE}(M,TM\otimes TM) \,
$$
in this case.

Before going to the isomorphism results, we recall that on an AE manifold $(M,g)$ the relevant part of the operator $L_g$ at infinity is given by
$$
L_{\mathbb{R}^m} \,\, := \,\, -\, \tfrac12 \, \Delta_{\mathbb{R}^m} \, + \, \tfrac{m-2}{2m} \, \rm{grad}_{\mathbb{R}^m}\circ \mbox{div}_{\mathbb{R}^m} \,.
$$
%and ${\mathcal L}$ are respectively
%$$
%L_0 \,\, : = \,\, -\, \frac12 \, \Delta \, + \, \frac{m-2}{2m} \, \nabla\mbox{div} \,,
%%\mbox{ and } \;\;\mathcal L_0=\Delta.
%$$
where the differential operators on the right hand side are computed with respect to the Euclidean metric, with the agreement that the Laplace operator here acts on each component of the vector field.

Now, exactly as in \cite[Proposition 13]{IMP}, we have
\begin{prop}
\label{iso AE}
Suppose that $(M,g)$ is an $m$-dimensional AE Riemannian manifold with
$2-m<-\nu<0$, then the vector Laplacian
$$
L_g : r^{-\nu}C^{k+2,\alpha}_{AE}(M, TM)\longrightarrow r^{-\nu-2}C^{k,\alpha}_{AE}(M,TM)
$$
is an isomorphism for every $k \in \mathbb{N}$.
\end{prop}
\begin{proof}
See \cite{IMP},
we just mention a misprint in the proof of Proposition 13,
the paper cited as the number [9] seems to be in fact the number [8] so \cite{Chr-Mur}
for us, especially Lemma 3.1 and Theorem 3.5 there, see  also \cite{Bartnikmass}.
\end{proof}

In the following, when at least one of the two components of $M_\ep$ is AE (and then $M_\ep$ itself is AE), we will systematically make use of the weighted functional setting introduced above in order to recover existence of solution to the equation \eqref{momentum}.

\subsection{Asymptotically Hyperbolic manifold}
An asymptotically hyperbolic (AH) manifold $(M,g)$ is the interior $M$ of a compact manifold  $\overline M$
with boundary $\partial M$, endowed with a metric $g=y^{-2}\overline g$, where $\overline g$ is a metric on $\overline M$
and $y$ is a defining function of the boundary $\partial M=y^{-1}(\{0\})$ such that $|dy|_{\overline g}=1$ on
$\partial M$. The terminology comes from the fact that, under this assumption, the sectional curvatures
of $(M,g)$ goes to $-1$ when approaching $\partial M$.

For a precise definition of the weighted H\"older spaces in the asymptotically hyperbolic context, we refer the reader to \cite{IMP}, with the only observation that the H\"older spaces denoted here by $y^{\nu}C^{k,\alpha}_{AH}$ coincide with $y^{\nu}\Lambda^{k,\alpha}_{AH}$ in \cite{IMP}. Hence, formally, a function (or tensor field) $u$ is in
$y^{\nu}C^{k,\alpha}_{AH}$ if $u\in C^{k,\alpha}_{loc}$ and, for all $l\leq k$, $\overline\nabla^{(l)}u=O(y^{\nu+l})$ near $\partial M$.

Since the model we want to reproduce, is the `horizontal' hyperbolic slice $\HN$ in Minkowski space, we will assume $\tau=m$, when at least one of the summands in the generalized connected sum is an AH manifold. Again, the Hamiltonian constraint suggests the assumption
$$
\mu\in y^{-1}C^{k,\alpha}_{AH} (M, TM \otimes TM).
$$
The isomorphism result in the AH context (see \cite{IMP} Proposition 16) reads
\begin{prop}
\label{iso AH}
Suppose that $(M,g)$ is an $m$-dimensional AH Riemannian manifold and suppose $0 < \nu < m+1$, then the vector Laplacian
%and $\nu'\in(-1,m)$, the mappings
$$
L_g:y^{\nu}C^{k+2,\alpha}_{AH}(M, TM)\longrightarrow y^{\nu}C^{k,\alpha}_{AH}(M,TM)
$$
%and
%$$
%\mathcal L:x^{\nu'}C^{k+2,\alpha}_{AH}\longrightarrow x^{\nu'}C^{k,\alpha}_{AH},
%$$
is an isomorphism for every $k \in \mathbb{N}$.
\end{prop}
\begin{proof}
See \cite{AndChDiss},  \cite{Lee:fredholm}, and also \cite{Gicquaud1}.
\end{proof}

In the following, when at least one of the two components of $M_\ep$ is AH (and then $M_\ep$ itself is AH), we will systematically make use of the weighted functional setting introduced above in order to recover existence of solution to the equation \eqref{momentum}.

\section{The momentum constraint: uniform {a priori} bound }
\label{uniform section}

In this section we will provide solutions to equation \eqref{momentum} with {\em a priori} bounds which do not depend on the parameter $\ep$. In analogy with \cite{Mazz 3}, this fact turns out to be crucial in the analysis of the Lichnerowicz equation, since it will allow us to control the squared pointwise norm of $\mu_{\gep} = \mu + D_{\gep} X$ in terms of the norm of $\mu$, which can be explicitly estimated.

It is clear that when $\ep$ tends to zero the geometry of the underlying manifold $M_\ep$ becomes closer and closer to a degenerate configuration, since the poly-neck size is shrinking down. On the other hand, it is clear that this phenomenon only takes place on the poly-neck region. Thus, to obtain the desired uniform bound, it is convenient to pass to a more sophisticated functional framework, by modifying the usual norms on the poly-neck region. We introduce then
a weighting function $\rho_\ep$ by setting $\rho_\ep \equiv 1$ on $M_i \setminus K^i_1$, $i=1,2$, and
$ \rho_\ep \, : = \, \ep \,
\cosh t$ for $(\log \ep) + 1 \, < \, t \, < \, -(\log \ep) - 1$, where $t$ is the poly-neck variable introduced in Section \ref{geom construction}. To complete the definition, we further assume $\rho_\ep$ to be a monotone radial smooth interpolation between
these regions.

Now, for $k \in \mathbb{N}$, $\al \in (0,1)$ and a weight parameter $\gamma \in \mathbb{R}$, we define the weighted
$\mathcal{C}^k$-norms and the weighted H\"{o}lder $(k,\al)$-seminorms of a tensor field $T$ on $(M_\ep,g_\ep)$. Let us set
\begin{eqnarray*}
\nor{T}{\mathcal{C}^k_{\gamma}} & := & \sum_{j=0}^k \,\,
\sup_{M_\ep} \,\,  \rho_\ep^{-\gamma + j} \, \cdot \,  \big|
\nabla^j T \big|_{\gep}  \\
\big[\nabla^k T \big]_{\alpha,\gamma} & := & \sup_{p \neq q} \,\,  \, \left|
\rho_\ep (p) \wedge \rho_\ep(q)\right|^{-\gamma + k} \, \cdot \,
\frac{\left| \nabla^k T(p) \, - \, \nabla^k T(q)
\right|_{\gep}}{d_{\gep} (p\, , \, q)^\al}  \, ,
\end{eqnarray*}
where $\nabla$ indicates the Levi-Civita connection of $\gep$, $
\left| \rho_\ep (p) \wedge \rho_\ep(q)\right|  $ is the minimum
between $\rho_\ep(p)$ and $\rho_\ep(q)$, and the conventions used in
(\ref{holder seminorm}) to define the H\"older ratio are still valid here. The weighted H\"older norm is then defined as
$$
\nor{T}{\mathcal{C}^{k,\al}_{\gamma}} \, := \,
\nor{T}{\mathcal{C}^k_{\gamma}} \, + \, \big[ \nabla^k
T \big]_{\al,\gamma} \,.
$$
The $C^{k,\al}_\gamma$-H\"older space is given by the tensor fields on $M_\ep$ with $(k,\al)$-regularity such that the above norm is well defined and finite.

We point out that for any fixed $\ep$ such a norm is equivalent to the standard $C^{k,\al}$-H\"older norm and thus the weighted spaces introduced here coincide with the usual $C^{k,\al}$-H\"older spaces, since the standard definition has been modified only on a compact region. Nevertheless, these new spaces will reveal to be the suitable ones in the analysis of the singular limit for $\ep\rightarrow 0$.

To treat all the possible situations at once, we denote by
$E^{k,\al}_\delta(M_\ep,TM_\ep)$ the space of vector fields on $M_\ep$ with $(k,\alpha)$-regularity endowed with $\delta$-weighted norm in the poly-neck region, $\delta \in \mathbb{R}$, and eventually with the suitable $\nu$-weighted norm on the AE or AH ends. Of course, the weight parameter $\nu$ at infinity is fixed according to the conventions of the previous section. Moreover, the fact that for any fixed $\ep$ the $C^{k, \al}_\delta$-norm is equivalent to the $C^{k,\al}$-norm insures that the isomorphism results still hold true in both the compact and the non compact case. In this context it is convenient to think of $L_{\gep}$ as acting between the spaces
\begin{eqnarray}
\label{weighted action} L_{\gep} \,\, : \,\, E^{k+2,\al}_\delta (M_\ep, TM_\ep) \,\, \longrightarrow
\,\, E^{k,\al}_{\delta - 2}(M_\ep,TM_\ep) \, ,
\end{eqnarray}
for a suitable weight $\delta \in \mathbb{R}$. We can now state the
following:
\begin{prop}
\label{uniform estimate} Let $X \, \in \,  \, E_\delta^{k+2,\al}(M_\ep, TM_\ep)$ and $W \, \in \, E_{\delta -2}^{k,\al}(M_\ep,TM_\ep)$
be vector fields satisfying the equation $L_{\gep} \, X \, = \, W$.
Moreover, suppose that $W$ is of the form $W \, = \, \sharp \,
\dv{\gep} \, \mu$, for some symmetric $2$-tensor $\mu$. Then for any $2-n < \delta < (2-n)/2$, there
exist a real number $\ep_0 \in (0,1)$ and a positive constant $C>0$ such that for every $\ep \in (0 , \ep_0)$
\begin{eqnarray}
\label{vector a priori}
\nor{X}{E^{k+2, \al}_\delta} & \leq & C \, \cdot \,
\nor{W}{E^{k,\al}_{\delta-2}} \, .
\end{eqnarray}
(Remember that $n$ is the codimension of $K$ in $M_i$, $i=1,2$ and
it is supposed to be at least $3$).
\end{prop}

\begin{proof} The proof is analogous to the prove of \cite[Proposition 2]{Mazz 3}. Nevertheless, we recall it here in oder to discuss the changes which are needed in the non compact cases, and possibly to simplify the argument. Here we only prove the uniform weighted $C^0$-bound, namely
\begin{eqnarray}
\label{vector C0}
\nor{X}{E^{0}_\delta} \,\, \leq \,\, C \, \cdot \,
\nor{W}{E^{0}_{\delta-2}} \, .
\end{eqnarray}
The estimate \eqref{vector a priori} will follow from local Schauder estimates for elliptic systems in divergence form (see \cite{Gia-Mar}, here we take advantage of the fact that $W= \sharp \dv{\gep} \mu$) combined with standard scaling arguments.

We proceed by contradiction. If \eqref{vector C0} does not hold, then it is possible to find a sequence $(\ep_j, X_j, W_j)$, $j \in \mathbb{N}$, such that
\begin{enumerate}
\item $\ep_j \, \rightarrow \, 0 \, ,$ \,\,\, $j \, \rightarrow \, + \infty$
\item $L_{\ep_j} \, X_j \, = \, W_j \, ,$\,\,\, $j \in \mathbb{N}$
\item $\nor{X_j}{E^0_\delta} \, = \, 1\, ,$ \,\,\,\, $j \in \mathbb{N}$ \quad and \quad $\nor{W_j}{E^{0}_{\delta-2}}  \rightarrow
  0\,,$\,\,\, $j \, \rightarrow \, + \infty$\, .
\end{enumerate}
For every $j\in \mathbb{N}$ we consider now a point $p_j \in M_{\ep_j}$ such that the pointwise norm $|X_j|(p_j)$ eventually multiplied by weighting function raised to the power indicated by the weight is greater or equal than $1-\ep_j$. For example, if the point $p_j$ lies in the region where $\rho_{j}$ is different from $1$, we are assuming that $\rho_j^{-\delta} |X_j| \,(p_j) \geq 1- \ep_j$. If $M_{\ep_j}$ has an AE end and the point $p_j$ is near infinity, we are assuming that $r^\nu |X_j|(p_j) \geq 1- \ep_j$ and so on. The fact that one can find such a point for each $j \in \mathbb{N}$ follows from the third hypothesis in the argument by contradiction. Depending on the behavior of the sequence $\{p_j \}$, we are going to distinguish several cases.

\smallskip

\textbf{Case 1a.} Suppose that up to a subsequence the points $p_j$'s lie in a compact region of $M_i \setminus K$ for some $i=1,2$. Without loss of generality we only consider the case $i=1$, since the other one is identical.

In analogy with \cite{Mazz 3}, we have that as $j \rightarrow + \infty$, the sequence of vector fields $X_j$ converges to a nontrivial weak solution $X$ of the  homogeneous problem
\begin{eqnarray}
  \label{Lg1}
  \left\{
    \begin{array}{lll}
      L_{g_1} \, X \,\, = \,\, 0 \quad \hbox{on $M_1\setminus K$} &  \\
      & \\
      |X|_{g_1}(\cdot) \,\, \leq \,\, A  \,\,\, |d_{g_1}(\, \cdot \, ,K)|^{\delta} \quad & %\\
%      &
%      %\hbox{where $r \, := \, d_{g_1}(\, \cdot \, ,K)$}
%      \\
%      |\nabla \, X|_{g_1} \, \leq \, A \, \cdot \, r^{\delta - 1} &
%      %\hbox{and $A,B > 0$ are positive constants}
    \end{array}
  \right.
  \end{eqnarray}
where $ d_{g_1}(\, \cdot \, ,K)$ represents the Riemannian distance to the submanifold $K$, and $A > 0$ is a positive constant. Without discussing all the details, we just recall that this follows from the fact that in our geometric construction the metrics $g_\ep$'s converge to the original metric $g_1$ on the compact subset of $M_1 \setminus K$. The growth prescription near $K$ is a direct consequence of the weighted setting introduced on the poly-neck region.

As it is easy to verify, the condition $2-n \, < \, \delta
$ guarantees that $X$ verifies the equation $L_{g_1}\, X \,  = \, 0$
on the whole $M_1$ in the sense of distributions, and thanks to the elliptic
regularity we deduce that $X$ can be extended through the singular set $K$ to a nontrivial smooth solution of $L_{g_1} X=0$. If $M_1$ is either AE or AH this contradicts the isomorphism results discussed in the previous section. If $M_1$ is compact, we integrate by parts obtaining
\begin{eqnarray*}
0  \quad = \quad \big( \, L_{g_1} \, X \,  , \, X \, \big)_{L^2} \quad
= \quad {}^{g_1} \nor{\DI{g_1}\, X}{L^2}^2 \, .
\end{eqnarray*}
Hence $X$ is a nontrivial conformal Killing vector field on $M_1$,
which is excluded by the nondegeneracy condition.

\smallskip

\textbf{Case 1b.} Suppose that $M_\ep$ is non compact and that, up to a subsequence, the points $p_j$ leave every compact subset of $M_\ep$. Again, without loss of generality, we assume that the points $p_j$'s of the subsequence lie on the piece of $M_\ep$ coming from $M_1$.

For every fixed compact subset $Q$ of $M_1$ we can suppose that $\rho_j^{-\delta}|X_j| \rightarrow 0$ on $Q$ when $j \rightarrow +\infty$, otherwise we are either in {\em Case 1a} discussed above or in {\em Case 2} or {\em 3} discussed below. We consider now a cut-off function $\varphi$ which is identically equal to 1 on $M_1 \setminus K^1_2$ and which vanishes on $K^1_1$, with the notations introduced in the introduction. Setting $Y_j := \varphi X_j$ and $Z_j := \varphi W_j \, + \, \big[L_j , \varphi  \big] \, X_j$, where $[\cdot , \cdot\cdot]$ is the formal commutator, we obtain $L_{g_1} Y_j = Z_j$. Moreover, for every $j \in \mathbb{N}$, both the vector fields $Y_j$ and $Z_j$ can be regarded as vector fields defined on $M_1$. Furthermore, using the second part of hypothesis 3 together with local Schauder estimates on $K^1_2$, it is easy to deduce that $Z_j \rightarrow 0$ in $E^0_0(M_1,TM_1)$, as $j \rightarrow 0$. On the other hand $\nor{X_j}{E^0_0(M_1 , TM_1)} \geq 1/2$ for large enough $j$. This contradicts the isomorphism results of the previous section.

\smallskip

\textbf{Case 2.} Suppose that up to a subsequence and without loss of generality the points $p_j$'s verify $d_{g_1} (p_j ,K) = \bigo{\ep_j}$, as $j \rightarrow + \infty$. This means that they are converging to $K$ with the largest velocity allowed.

Performing the same blow-up as in \cite[Proposition 2]{Mazz 3}, we have that the vector fields $X_j$ converge to a nontrivial weak solution $X$ of the homogeneous problem
 \begin{eqnarray}
  \label{Lrksn}
  \left\{
    \begin{array}{lll}
      L_{\mathbb{R}^k \times {\mathbb{S}}^n} \, X \,\, = \,\, 0 \quad \hbox{on $\mathbb{R}_z^k \times \mathbb{S}_{t, \theta}^n$} & \\
      & \\
      |X|_{\mathbb{R}^k \times {\mathbb{S}}^n} \,\, \leq \,\, B \,\,\, \left( \cosh t \right)^{\delta} \quad
      &\\
%      &\\
%      %\hbox{where $C>0$ is a positive constant}
%      |\nabla \, X |_{\mathbb{R}^k \times {\mathbb{S}}^n} \, \leq
%      \, B \, \cdot \, \left( \cosh t\right)^{\delta - 1} &
    \end{array}
  \right.
  \end{eqnarray}
  where $B>0$ is a positive constant and $\mathbb{S}^n$ denotes the $n$-dimensional (Riemannian) Schwarzschild space, $n \geq 3$. For sake of completeness we recall that the space $\mathbb{S}^n$ is diffeomorphic to the standard cylinder $\mathbb{R}_t \times S^{n-1}_\theta$ and it is endowed with the metric
$$
g_{\mathbb{S}^n} \,\, = \,\, \big[ \cosh \big( \tfrac{n-2}{2} t\big)  \, \big]^{\frac{4}{n-2}} \, \cdot \, dt^2 + g_{S^{n-1}} \, .
$$
We also point out that the expressions $L_{\mathbb{R}^k \times \mathbb{S}^n}$ and $|\,\cdot\,|_{\mathbb{R}^k \times \mathbb{S}^n}$ indicate respectively the vector Laplacian and the pointwise norm of the product metric $g_{\mathbb{R}^k}  +  \, g_{\mathbb{S}^n}$.

In order to simplify the argument presented in \cite{Mazz 3} and getting the desired contradiction, we split the vector field $X$ into the $T\mathbb{R}^k$ component and the $T\mathbb{S}^n$ component. Thus, we write $X=U+V$, where
\begin{eqnarray*}
U
%(z,t, \theta)
\,\, = \,\, U^j
(z, t, \theta)
\cdot   {e}_j
%\frac{\partial}{\partial z^j}
& \quad \hbox{and} \quad  & V
%(z,t,\theta)
\,\, = \,\, V^{t}(z, t, \theta) \,\, \frac{\partial}{\partial t} \, + \, V^\mu (z, t, \theta) \,\,\frac{\partial}{\partial \theta^\mu} \,.
\end{eqnarray*}
Here $\{e_1, \ldots , e_k \}$ is the standard basis on $T\RK \simeq \RK$.
The action of $L_{\mathbb{R}^k \times \mathbb{S}^n}$ is then given by

\medskip

$
{L}_{\RK \times \EN}\, {X} \,\,\, = \,\,\,$

\vspace{-0.2cm}

$$
  \left( \begin{array}{cc}
  \quad \quad L_\RK \,  -  \tfrac{1}{2} \, \Delta_\EN
 \quad & - \,
\big(\tfrac{m-2}{2\,m}\big)\,\, {\rm grad}_{\RK} \circ \dv{\EN} \\
 - \,
\big(\tfrac{m-k}{k m} \big)
\, {\rm grad}_\RK  \circ \, \dv{\RK}
 &  \\
& \\
    -
\big(\tfrac{m-2}{2\,m}\big) \, {\rm grad }_{\EN} \circ \dv{\RK}   &  L_{\EN} \,
 - \tfrac{1}{2}   \Delta_{\RK}  \\
 & - \big(\tfrac{m-n}{n m}\big) \, {\rm
grad}_{\EN} \circ \dv{\EN}  (\, \cdot \,)
 \\
  \end{array} \right)  \,\,
  \left( \begin{array}{c}
    \\
     U \\
    \\
     V \\
    \\
  \end{array} \right)
$$
%\begin{eqnarray*}
%\left[L_{\RK \times \EN} \, X \right]^{T\RK} & = & L_\RK \, U \, - \,
%\big(\tfrac{m-k}{k m} \big)
%\, \, {\rm grad}_\RK  \circ \, \dv{\RK} U \\
%&  & - \,\,  \tfrac{1}{2} \, \Delta_\EN U \, - \,
%\big(\tfrac{m-2}{2 m} \big) \, \, {\rm grad}_{\RK} \circ \, \dv{\EN}V  \nonumber
%\end{eqnarray*}
%\vspace{-0.4cm}
%\begin{eqnarray*}
%\left[L_{\RK \times \EN} \, X \right]^{T\,\EN} &=& L_\EN \, V \, -
%\, \big(\tfrac{m-n}{n m} \big)
%\,  \, {\rm grad}_\EN \circ \, \dv{\EN} V \\
%& & - \, \tfrac{1}{2} \,  \Delta_\RK V \, - \,
%\big(\tfrac{m-2}{2 m } \big) \,  \, {\rm grad}_{\EN}
%\circ \, \dv{\RK}U  \, ,\nonumber
%\end{eqnarray*}
where $\Delta_{g}\, W$ indicates the (negative
definite) Laplace-Beltrami operator of a Riemannian metric $g$ applied to the components of the vector field $W$.

We observe that, by the elliptic regularity, the components $U^j, V^t$ and $V^\mu$ are smooth functions which verify the decay prescription when $|t| \rightarrow + \infty$ and which are uniformly bounded in the $z$-direction. To proceed with the analysis, it is then convenient to consider these components as $C^{2, \al}_\delta$-valued tempered distributions, namely as elements in
$$
\mathcal{S}' \big(\RK,C^{2, \al}_\delta(\mathbb{S}^n) \big) \,.
$$
Now, we are going to take the Fourier transform along the $z$ variable
\begin{eqnarray*}
\hat{U}
%(z,t, \theta)
\,\, = \,\, \hat{U}^j
(\zeta , t, \theta)
\cdot e_j & \quad \hbox{and} \quad  & \hat{V}
%(z,t,\theta)
\,\, = \,\, \hat{V}^{t}(\zeta, t, \theta) \,\, \frac{\partial}{\partial t} \, + \, \hat{V}^\mu (\zeta, t, \theta) \,\,\frac{\partial}{\partial \theta^\mu} \,\, .
\end{eqnarray*}
Notice that $\hat U$ and $\hat V$ verify the decay prescriptions
$$
|\hat{U}|_{\RK} (\zeta, t ,\theta) \,\, \leq \,\, B \, (\cosh t)^\delta \quad  \,\, \hbox{and} \,\, \quad |\hat V|_{\mathbb{S}^n} (\zeta, t ,\theta)\,\, \leq \,\, B \, (\cosh t)^{\delta}
$$
along with their derivatives.

We have now at hand the family of operators $\hat L
%_{\RK \times \mathbb{S}^n}
(\zeta)$ acting on $\hat X (\zeta, t, \theta) = \hat U (\zeta, t ,\theta) + \hat V (\zeta, t, \theta)$ in the following way
$$
\hat{L}(\zeta) \, \hat{X} \,\,\, = \,\,\,
  \left( \begin{array}{cc}
    -  \tfrac{1}{2} \, \Delta_\EN
\,  + \, \tfrac{1}{2}  \, |\zeta|^2  \quad & - i \,
\big(\tfrac{m-2}{2\,m}\big)\,\, \dv{\EN}(\,\cdot \,) \, \,
\zeta    \\
 \quad \quad  + \, \big(\tfrac{m-2}{2\,m}\big)\,\, \langle \zeta ,
\, \cdot \, \rangle
\,\, \zeta  \, &  \\
& \\
    -  i
\big(\tfrac{m-2}{2\,m}\big) \, {\rm grad }_{\EN}
\langle \zeta , \, \cdot \, \rangle   &  L_{\EN} \,
 +  \tfrac{1}{2}   |\zeta|^2  \\
 & - \big(\tfrac{m-n}{n m}\big) \, {\rm
grad}_{\EN} \circ \dv{\EN}  (\, \cdot \,)
 \\
  \end{array} \right)  \,\,
  \left( \begin{array}{c}
    \\
    \hat U \\
    \\
    \hat V \\
    \\
  \end{array} \right)
$$
Clearly $\hat L (\zeta) \hat X = 0$. If we show that, for every $\zeta \in \RK$, $\hat L (\zeta) \, \hat X = 0$ implies $\hat X = 0$, we will also have that $X$ must be identically $0$, reaching a contradiction. To show that this is actually the case, we are going to integrate by parts the expression
$$
0 \,\,\, = \,\,\, \big( \hat L (\zeta) \, \hat X , \hat X  \big)_{L^2} \, .
$$
A direct computation shows that
\begin{eqnarray*}
0 & = &  \, \int_{\EN} \hbox{$\sum^k_{j=1}$} \, \big|  {\rm
grad_{\EN}} \, \hat U^j \big|_\EN^2  + \, \big(
\tfrac{m-2}{m} \big) \,  \big| \langle \zeta, \hat U
\rangle \big|^2  + |\zeta|^2 \,  |\hat U |^2_\RK  \,
\, \dvol{\EN} \\
&  &+  \int_{\EN} 2 \, \big| \DI{\EN} \hat V \big|_\EN^2  +
\, 2 \, \big( \tfrac{m-n}{ n m} \big)  \big| \dv{\EN
} \hat V \big|_\EN^2  + |\zeta|^2 \, |\hat V|_\EN^2  \, \, \dvol{\EN} \, \\
& & - \lim_{T\rightarrow+\infty} \int_{S^{n-1}}  \left.    \hbox{$\sum_{j=1}^k$} \,  \hat U^j \, \partial_t \hat U^j \, \big[ \cosh{\tfrac{n-2}{2}t} \big]^{2} \, \right|^{+T}_{-T} \,\, \dvol{S^{n-1}} \\
& & - \lim_{T\rightarrow+\infty}  2 \, \big(\tfrac{m-n}{nm}\big) \int_{S^{n-1}}  \left.    \hat V^t \, \dv{\EN}\hat V \, \big[ \cosh{\tfrac{n-2}{2}t} \big]^{\frac{2n}{n-2}} \, \right|^{+T}_{-T} \,\, \dvol{S^{n-1}} \\
& & - \lim_{T\rightarrow+\infty}  2 \int_{S^{n-1}}  \left.    \DI{\EN} \hat V \, ( \, \hat V , \partial_t\,) \, \big[ \cosh{\tfrac{n-2}{2}t} \big]^{2} \, \right|^{+T}_{-T} \,\, \dvol{S^{n-1}} \\
& & - \lim_{T\rightarrow+\infty}  \, i \, \big(\tfrac{m-2}{m}\big) \int_{S^{n-1}}  \left.    \hat V^t \, \langle \zeta , \hat U \rangle \, \big[ \cosh{\tfrac{n-2}{2}t} \big]^{\frac{2n}{n-2}} \, \right|^{+T}_{-T} \,\,\dvol{S^{n-1}} \,.
\end{eqnarray*}
Using the decay prescription and remembering that $\delta < (2-n)/2$ it is immediate to check that first three boundary terms tend to zero as $T \rightarrow + \infty$. Moreover, when $\zeta = 0 $, also the last boundary term gives no contribution and we deduce immediately that $\hat U$ is constant and $\hat V$ is a conformal Killing vector field on $\EN$. Since they decay at infinity, they must vanish everywhere, leading us to a contradiction.

Hence we suppose now $\zeta \neq 0$ and we show that also in this case the contribution of the last boundary term is zero. First we observe that the component $\hat V^t$ is controlled by $(\cosh t)^{\delta -1}$, whereas {\em a priori} the bound on $|\hat U|_{\RK}$ is just given by $(\cosh t)^{\delta}$. Since the third factor in the last integral goes like $(\cosh t)^n$, this is in general not sufficient to guarantee the vanishing of the entire boundary term in the limit for $T \rightarrow + \infty$. To overcome this difficulty, we are going to show that the decay of $\hat U$ is better than expected. In fact, from $\hat L (\zeta) \hat X = 0$ we have that
$$
 \Delta_\EN \, \hat U
\,  - \,  \, |\zeta|^2 \,  \hat U \,  - \, \big(\tfrac{m-2}{m}\big)\,\, \langle \zeta ,
\hat U \rangle \, \zeta \,\,\, = \,\,\, - i \,
\big(\tfrac{m-2}{m}\big)\,\, (\dv{\EN} \hat V ) \,
\zeta \,.
$$
Multiplying this equation by a vector $\eta 	\in \RK$ with $\langle \eta, \zeta \rangle = 0$, we obtain
$$
 \Delta_\EN \, \langle \hat U , \eta \rangle
\,  - \,  \, |\zeta|^2 \,  \langle \hat U , \eta \rangle \,\,\, = \,\,\, 0 \, .
$$
Since $\hat U$ decays at infinity, we can use the maximum principle to deduce that $\langle \hat U , \eta \rangle =0$ everywhere on $\EN$. As a consequence we write the vector field $\hat U (\zeta, t ,\theta)$ as $u (\zeta, t ,\theta) \, \zeta$, where $u$ is a suitable function verifying the decay prescriptions induced by $\hat U$. The equation above becomes then
$$
 \Delta_\EN \, u
\,   - \, \big(\tfrac{2m-2}{m}\big)\,\,  |\zeta|^2 \,
u   \,\,\, = \,\,\, - i \,
\big(\tfrac{m-2}{m}\big)\,\, (\dv{\EN} \hat V ) \,.
$$
In order to split this equation into its real part and its pure imaginary part, we set $u = \phi + i \psi$ and $\dv{\EN} \hat V = a+ib$, where, according to the decay prescription on $\hat V$, both the functions $a$ and $b$ are bounded by $(\cosh t)^{\delta - 1}$. We obtain then
\begin{eqnarray*}
\Delta_\EN \, \phi
\,   - \, \big(\tfrac{2m-2}{m}\big)\,\,  |\zeta|^2 \,
\phi  & = &  \,
\big(\tfrac{m-2}{m}\big)\,\, b \\
\Delta_\EN \, \psi
\,   - \, \big(\tfrac{2m-2}{m}\big)\,\,  |\zeta|^2 \,
\psi  & = &  \,
\big(\tfrac{2-m}{m}\big)\,\, a \,.
\end{eqnarray*}
Now, using separation of variables and translating the problem into an infinite dimensional system of ODE's, it is easy to obtain explicit solutions for this Bessel type equations. Moreover, both the solutions $\phi$ and $\psi$ are now of the same order of the right hand side, as $t \rightarrow + \infty$. Coming back to $\hat U$, we have that $|\hat U|_\RK$ is bounded by $(\cosh t)^{\delta -1}$. Hence also the last boundary term is vanishing in the limit for $T \rightarrow + \infty$. The integration by parts is then completely justified and it is now straightforward to deduce that also when $\zeta \neq 0$ the vector fields $\hat U$ and $\hat V$ are forced to be trivial, which is a contradiction.

\smallskip

\textbf{Case 3.} Suppose that up to a subsequence and without loss of generality the points $p_j$'s still converge to a point $p_\infty$ which lies in $K$, but with a lower velocity than in the previous case. This can be phrased by saying that $d_{g_1}(p_j, K) \, / \,\ep_j \, \rightarrow \, +\infty$, as $j \rightarrow + \infty$ (notice that, up to pass to a new subsequence, we can suppose that all the $p_j$'s lie in the side of $M_\ep$ which comes from $M_1$). Following \cite[Proposition 2]{Mazz 3}, we rescale our sequence of problems with the rate of convergence of the $p_j$'s, so that the blow-up limit is now given by
 \begin{eqnarray}
  \label{Lrkrn}
  \left\{
    \begin{array}{ll}
      L_{\mathbb{R}^k \times {\mathbb{R}^n}} \, X \, = \, 0 \quad & \hbox{on $\mathbb{R}_z^k \times \left(\mathbb{R}_x^n \setminus \{0\} \right)$ } \\
      & \\
      |X|_{\mathbb{R}^k \times {\mathbb{R}}^n} \, \leq \, C \, \, |x|^{\delta} \quad &
      %\hbox{where $D>0$ is a positive constant }\\
%      \\ & \\
%      |\nabla \, X|_{\mathbb{R}^k \times {\mathbb{R}}^n} \, \leq \, C \, \cdot \, |x|^{\delta-1} \quad &
    \end{array}
  \right.
  \end{eqnarray}
where $C>0$ is a positive constant and $X$ is nontrivial.

This last case can be treated easily first by showing that the condition $2-n < \delta$ allows one to extend smoothly the solution through the singular set $\RK \times \{0\}$ in analogy with the {\em Case 1.a}, and then using the decay condition at infinity, with $\delta < (2-n)/2$, as in {\em Case 2} to justify the integration by parts which implies the vanishing of $X$ and thus the contradiction.
\end{proof}

%\vspace{3cm}

%Nevertheless, we are going to use here a slightly different and more direct approach. We start by studying the case $k=0$, corresponding to the pointwise connected sum. In this case the problem reduces to
%\begin{eqnarray*}
%  \left\{
%    \begin{array}{ll}
%      L_{\mathbb{R}^n } \, X \, = \, 0 \quad & \hbox{on $ \mathbb{R}_x^n \setminus \{0\} $ } \\
%      & \\
%      |X|_{ {\mathbb{R}}^n} \, \leq \, C \, \, |x|^{\delta} \quad &
%    \end{array}
%  \right.
%  \end{eqnarray*}
%To analyze this operator it is convenient to pass to cylindrical coordinates setting $t = \log|x|$ and $\theta = x/|x|$

\section{The energy constraint}

In this section we conclude the proof of Theorem \ref{teo 1} by showing that the Lichnerowicz equation can be solved
\begin{eqnarray}
\label{licn}
\Delta_{\gep} \, u \, - \, c_m \, R_{\gep} \, u \, + \, c_m \, |\,
\mu_{\ep} |_{\gep}^2 \, u^{- \frac{3m-2}{m-2}} \, - \, c_m \,
\tfrac{m-1}{m} \, \tau^2 \, u^{\frac{m+2}{m-2}} & = & 0 \, ,
\end{eqnarray}
provided the parameter $\ep$ is sufficiently small. Before starting, we recall that the terms $\mu_\ep$ is the sum of the `proto TT-tensor' $\mu$ defined in of Section \ref{geom construction} and the correction term $\DI{\gep} X$, where the vector field $X$ solves the equation \eqref{momentum} and hence by Proposition \ref{uniform estimate} is $\ep$-uniformly bounded in terms of $\mu$.

In analogy with \cite{Mazz 3}, the solution to \eqref{licn} will be obtained by perturbation. Since in Theorem \ref{teo 1} we claimed the convergence of the new solutions to the original ones, it is natural to linearize our equation about the constant 1. Hence we set $v:=u-1$ and we consider the linear operator
\begin{eqnarray}
\mathcal{L}_{\gep} & = & \Delta_{\gep} \, - \, \chi_1 \,
\left(|\mu_1|^2_{g_1} \, + \, {\tau^2}/{m}\right) \, - \, \chi_2 \, \left( |\mu_2|^2_{g_2} \, + \, \tau^2/m  \right) \, ,
\end{eqnarray}
where $\chi_1 := \chi$, $\chi_2 := 1-\chi$ and the smooth cut-off $\chi$ has been defined in Section \ref{geom construction}. In the following $\mathcal{L}_{\gep}$ will be referred as the linearized Lichnerowicz operator. For sake of completeness we recall that $\Delta_{\gep}$ can be written on the poly-neck as
\begin{eqnarray*}
\Delta_{\gep} & = & u_\ep^{-\frac{4}{n-2}} \big[ \, \partial^2_t \, + \, (n-2) \tanh \big( \tfrac{n-2}{2} t\big) \, \partial_t \, + \, \Delta_{S^{n-1}}  \, + \, u_\ep^{\frac{4}{n-2}} \Delta_{g_K} \, + \, P_\ep  \, \big] \, ,
\end{eqnarray*}
where $P_\ep$ is a second order partial differential operator whose coefficients are bounded by $\rho_\ep$. Finally, collecting all the remainders, we define the error term $F_\ep (v)$ by
\begin{eqnarray}
F_{\ep}(v) & := & - \, c_m \, (R_{g_1} - R_{\gep}) \, \chi_1 \, + \, c_m
\, (|\mu_1|^2_{g_1}  - |\mu_\ep|^2_{\gep}) \, \chi_1  \nonumber\\
& & - \, c_m \, (R_{g_2} - R_{\gep}) \, \chi_2 \, + \, c_m
\, (|\mu_2|^2_{g_2}  - |\mu_\ep|^2_{\gep}) \, \chi_2 \nonumber\\
& & - \,c_m \, (R_{g_1} - R_{\gep}) \, \chi_1 \, v \, - \, b_m
\, (|\mu_1|^2_{g_1}  - |\mu_\ep|^2_{\gep}) \, \chi_1 \,v \\
& & - \, c_m \, (R_{g_2} - R_{\gep}) \, \chi_2 \,v \, - \, b_m
\, (|\mu_2|^2_{g_2}  - |\mu_\ep|^2_{\gep}) \, \chi_2  \, v \nonumber\\
& & - \, c_m \, |\mu_\ep|^2_{\gep} \,  h(v) \, + \, c_m \,
\tfrac{m-1}{m} \, \tau^2 \,  f(v) \, , \nonumber
\end{eqnarray}
\noindent where $c_m = (m-2) / [4 (m-1)]$, $b_m =  c_m
(3m-2)/(m-2)$ and
\begin{eqnarray*} h(v) & := & \left[ (1+v)^{-
\frac{3m-2}{m-2}} - 1 +
\big( \tfrac{3m-2}{m-2}\big) v \right] \\
f(v) & := & \left[ (1+v)^{\frac{m+2}{m-2}} - 1 -
\big(\tfrac{m+2}{m-2}\big) v \right] \, .
\end{eqnarray*}
Hence both $h$ and $f$ are $\bigo{|v|^2}$.

In order to solve \eqref{licn}, we first prove  invertibility and {\em a priori} estimates for $\mathcal{L}_{\gep}$ and then we show that the fixed point problem
\begin{eqnarray}
\label{fixed}
v & = & \mathcal{L}_{\gep}^{-1} \circ F_\ep (v)
\end{eqnarray}
admits a solution.

\subsection{Analysis of the linearized Lichnerowicz operator}

We consider the following linear problem
\begin{eqnarray}
\label{lin licn}
\mathcal{L}_{\gep} \, v & = & w \quad \quad  \hbox{on $\,\,M_\ep$}\,.
\end{eqnarray}
Concerning the existence of solutions, we observe that in the compact case this is a consequence of the non vanishing of $\tilde \Pi_1$ and $\tilde \Pi_2$. In fact this implies that the operator $\mathcal{L}_{\gep}$ is negative definite, thus injective on $L^2(M_\ep)$. Since it is also self adjoint, we deduce by the Fredholm alternative that it is also surjective. Moreover, by elliptic regularity, we have that if the right hand side is in $C^{0,\al}(M_\ep)$, then the solution belongs to $C^{2,\al}(M_\ep)$. If one of the summands is either AE or AH, the existence of solutions is guaranteed by the following issues without any further assumption on $\tilde \Pi_i$, $i=1,2$ (the references for these results are the same as for Proposition \ref{iso AE} and \ref{iso AH}).
\begin{prop}
Suppose that $(M,g)$ is an $m$-dimensional AE Riemannian manifold with
$2-m<-\nu<0$, then the linearized Lichnerowicz operator
$$
\mathcal{L}_g : r^{-\nu}C^{k+2,\alpha}_{AE}(M)\longrightarrow r^{-\nu-2}C^{k,\alpha}_{AE}(M)
$$
is an isomorphism for every $k \in \mathbb{N}$.
\end{prop}
\begin{prop}
Suppose that $(M,g)$ is an $m$-dimensional AH Riemannian manifold with $0 < \nu < m+1$, then, if $\nu'\in(-1,m)$, the linearized Lichnerowicz operator
$$
\mathcal L:y^{\nu'}C^{k+2,\alpha}_{AH}(M)\longrightarrow y^{\nu'}C^{k,\alpha}_{AH}(M),
$$
is an isomorphism for every $k \in \mathbb{N}$.
\end{prop}
We pass now to establish the {\em a priori} bound for solutions to \eqref{lin licn}. To treat all the possible situations at once, we denote by $F^{k,\al}_\gamma (M_\ep)$ the space of functions $M_\ep$ with $(k,\alpha)$-regularity endowed with $\gamma$-weighted norm in the poly-neck region, $\gamma \in \mathbb{R}$, and eventually with the suitable $\nu$ or $\nu'$-weighted norm on the AE or AH ends. The linearized Lichnerowicz operator is naturally defined between the spaces
\begin{eqnarray*}
%\label{weighted action}
\mathcal{L}_{\gep} \,\, : \,\, F^{k+2,\al}_\gamma (M_\ep) \,\, \longrightarrow
\,\, F^{k,\al}_{\gamma - 2}(M_\ep) \, ,
\end{eqnarray*}
for $\gamma \in \mathbb{R}$. We can now state the
following:
\begin{prop}
\label{uniform estimate licn} Let $v  \in  F_\gamma^{k+2,\al}(M_\ep)$ and $w  \in  F_{\gamma -2}^{k,\al}(M_\ep)$
be functions satisfying the equation $\mathcal{L}_{\gep} \, v \, = \, w$. Then, for any $2-n < \gamma < 0$, there
exist a real number $\ep_0 \in (0,1)$ and a positive constant $C>0$ such that for every $\ep \in (0,\ep_0)$
\begin{eqnarray}
%\label{vector a priori}
\nor{v}{F^{k+2, \al}_\gamma} & \leq & C \, \cdot \,
\nor{w}{F^{k , \al}_{\gamma-2}} \, .
\end{eqnarray}
(Remember that $n$ is the codimension of $K$ in $M_i$, $i=1,2$ and
it is supposed to be at least $3$).
\end{prop}
\begin{proof}
As in Proposition \ref{uniform estimate} the difficult part is to provide the uniform $C^0_\gamma$ bound. Hence we will focus here on the estimate
\begin{eqnarray}
\label{bound}
\nor{v}{F^{0}_\gamma} & \leq & C \, \cdot \,
\nor{w}{F^{0}_{\gamma-2}} \, .
\end{eqnarray}
As a first step we produce a local version of the {\em a priori estimate}. Assuming the hypothesis of the theorem, we claim that there exist a real number $\beta= \beta(n,\gamma)>0$ and a positive constant $B=B(n,\gamma)>0$ such that for every $\ep \in (0, e^{-\beta})$
\begin{eqnarray}
\label{local a. p. estimate}
\nor{v}{\mathcal{C}^0_{\gamma}(T^\ep_\beta)} & \leq &
B \,\, \big[ \,\, \nor{w}{\mathcal{C}^0_{\gamma -
2}(T^\ep_\beta)} \, + \,
\nor{v}{\mathcal{C}^0_{\gamma}(\partial T^\ep_\beta)} \, \big]\, ,
\end{eqnarray}
where $T^\ep_\beta$ is the portion of the poly-neck where $\beta + \log \ep \leq t \leq -\beta -\log \ep$.

To prove this claim, we observe that the approximate solution metrics are the same as the ones used in \cite{Mazz} in order to produce the generalized connected sum of constant scalar curvature metrics. Hence, we recover from \cite[Lemma 3]{Mazz} the existence of barrier function for $\Delta_{\gep}$. Combining this with the maximum principle one can easily deduce that if $\Delta_{\gep} v = f$, with $f \in C^0$, then for every $2-n < \gamma <0$ there exist a real number $\bar\beta=\bar\beta(n,\gamma)>0$ and a positive constant $A=A(n, \gamma)>0$ such that for
every $\beta > \bar \beta$ and every $\ep \in (0, e^{-\beta})$
\begin{eqnarray*}
\nor{v}{\mathcal{C}^0_{\gamma}(T^\ep_\beta)} & \leq &
A \,\, \big[ \,\, \nor{f}{\mathcal{C}^0_{\gamma -
2}(T^\ep_\beta)} \, + \,
\nor{v}{\mathcal{C}^0_{\gamma}(\partial T^\ep_\beta)} \, \big]\, .
\end{eqnarray*}
To obtain the claim, we let now
$$
f \,\, = \,\,  w \,  + \,\big[ \, \chi_1 \,
\left(|\mu_1|^2_{g_1} \, + \, {\tau^2}/{m}\right) \, + \, \chi_2 \, \left( |\mu_2|^2_{g_2} \, + \, \tau^2/m  \right) \, \big] \, v
$$
and we observe that for large enough $\beta$, we have
\begin{eqnarray*}
\sup_{T^\ep_\be} \left|\rho^{-\gamma +2}_\ep \, |\mu_1|^2_{g_1} \, v
\, \right| & \leq & \tfrac{1}{4} \, \sup_{T^\ep_\be}
\left|\rho_\ep^{-\gamma} \, v \, \right| \\
\sup_{T^\ep_\be} \left|\rho^{-\gamma +2}_\ep \, \left( \tau^2/m
\right) \, v \, \right| & \leq & \tfrac{1}{4} \, \sup_{T^\ep_\be}
\left|\rho_\ep^{-\gamma} \, v \, \right| \, ,
\end{eqnarray*}
for every $\ep \in (0, e^{-\beta})$. The estimate (\ref{local a. p. estimate}) follows at once.

To obtain \eqref{bound} we argue by contradiction as in the proof of Proposition \ref{uniform estimate}. If \eqref{bound} does not hold, we can consider a sequence of counterexamples $(\ep_j, v_j, w_j)$, $j \in \mathbb{N}$, such that
\begin{enumerate}
\item $\ep_j \, \rightarrow \, 0 \, ,$ \,\,\, $j \, \rightarrow \, + \infty$
\item $\mathcal{L}_{\ep_j} \, v_j \, = \, w_j \, ,$\,\,\, $j \in \mathbb{N}$
\item $\nor{v_j}{F^0_\gamma} \, = \, 1\, ,$ \,\,\,\, $j \in \mathbb{N}$ \quad and \quad $\nor{w_j}{F^{0}_{\gamma-2}}  \rightarrow 0\,,$\,\,\, $j \, \rightarrow \, + \infty$\, .
\end{enumerate}
As in Proposition \ref{uniform estimate} we consider the sequence of points $p_j \in M_{\ep_j}$, $j \in \mathbb{N}$ such that the pointwise norm $|v_j|(p_j)$ eventually multiplied by weighting function raised to the power indicated by the weight is greater or equal than $1-\ep_j$.
Depending on the behavior of the sequence $\{p_j \}$, we are going to distinguish several cases.

\smallskip

\textbf{Case 1.} The first case is when up to a subsequence the points $p_j$ lie in a compact region of say $M_1\setminus K$ (the case of $M_2 \setminus K$ is analogous). Hence, it is not hard to dee that the functions $v_j$ tend to a nontrivial distributional solution $v_\infty$ of the homogeneous problem $\mathcal{L}_{g_1} v_{\infty} = 0$ on $M_1 \setminus K$, which in addition satisfies the growth prescription $|\, v_\infty| \leq C \, |d_{g_1} (\, \cdot \, , K)|^{\gamma}$. Since $2-n < \gamma$, $v_\infty$ extends smoothly through $K$. On the other hand, we have already seen that the linearized Lichnerowicz operator about the metric $g_1$ is injective in both the compact and the non compact case. Hence, $v_\infty$ must be trivial, which is a contradiction.

\smallskip

\textbf{Case 2.} The second case is when $M_\ep$ is non compact and, up to a subsequence, the points $p_j$'s leave every compact subset of $M_\ep$. Again, without loss of generality, we assume that the points $p_j$'s of the subsequence lie on the piece of $M_\ep$ coming from $M_1$. In this case the contradiction is reached by the same argument used in {\em Case 1b} of Proposition 5, mutatis mutandis.

\smallskip

\textbf{Case 3.} The remaining case is when up to a subsequence the points $p_j$'s lie in the region $T^{\ep_j}_\be$. Again, in this case the functions $v_j$ tend to a distributional solution of $\mathcal{L}_{g_i} v_{\infty} = 0$ on $M_i \setminus K$ verifying the prescription $|\, v_\infty| \leq C \, |d_{g_i} (\, \cdot \, , K)|^{\gamma}$, for $i=1,2$. The fact that $v_\infty$ is nontrivial is now a consequence of the local {\em a priori} bound \eqref{local a. p. estimate}. In fact in correspondence to the $p_j$'s, we can find another sequence of points $q_j$'s lying in $\partial T^{\ep_j}_\beta$, where the quantities $|v_j|$'s are uniformly bounded from below. Notice that in the limit, the loci $\partial T^{\ep_j}_\beta$ correspond to the locus $\{d_{g_1} (\, \cdot \, , K) = e^{-\beta}\} \cup \{d_{g_2} (\, \cdot \, , K) = e^{-\beta}\}$. To reach the contradiction it is now sufficient to repeat the same argument as in the {\em Case 1} above.
\end{proof}

\subsection{Fixed point argument}

In this subsection, taking advantage of Proposition \ref{uniform estimate licn} we will produce a solution to equation \eqref{licn} by solving the fixed point problem \eqref{fixed}.

As a first step we show that $\mathcal{L}_{\gep} \circ F_\ep$ viewed as a mapping from the space of continuous functions on $M_\ep$ in itself sends a small ball centered at the origin in itself. Having this result, it is not difficult to show that $\mathcal{L}_{\gep}^{-1} \circ F_{\ep}$ is a contraction on this ball. This will provide automatically the fixed point and the control on the size of the correction.

We start with the estimate of the error term. To simplify the notations we decompose $F_\ep (v)$ as
$$
F_\ep(v) \,\, = \,\, F_\ep^{(0)} \, + \, F_\ep^{(1)} (v) \, + \, F_\ep^{(2)} (v,v) \, ,
$$
where the three summands on the right hand side correspond respectively to the zero order term of $F_\ep(v)$, the terms which have a linear dependence on $v$ and to the terms which depend quadratically on $v$. Here the crucial part is the estimate of the pure error term
\begin{eqnarray*}
F^{(0)}_{\ep} & := & -  \, c_m \, (R_{g_1} - R_{\gep}) \, \chi_1 \, + \, c_m
\, (|\mu_1|^2_{g_1}  - |\mu_\ep|^2_{\gep}) \, \chi_1  \nonumber\\
& & - \, c_m \, (R_{g_2} - R_{\gep}) \, \chi_2 \, + \, c_m
\, (|\mu_2|^2_{g_2}  - |\mu_\ep|^2_{\gep}) \, \chi_2 \,.
\end{eqnarray*}
Recalling that the approximate solution metrics are the same used by the second author in \cite{Mazz}, we recover from \cite[Proposition 2]{Mazz} the estimate of the scalar curvature error. More precisely, we have that on the poly-neck region $\{\log \ep \leq t \leq -\log \ep  \}$
$$
\big | \, \chi_1 \, (R_{g_1} - R_{\gep}) \, + \,\chi_2 \, (R_{g_2} - R_{\gep}) \, \big | \,\, \leq \,\, C_1 \,  \ep^{n-2} \, \rho_\ep^{1-n} \,,
$$
for some positive constant $C_1>0$ independent of $\ep \in (0,1)$. Now, using Proposition \ref{uniform estimate}, we show that the same bound holds for the $TT$-tensor error.
We claim that
\begin{eqnarray}
\label{mu bound}
\big | \, \chi_1 \, (|\mu_1|^2_{g_1}  - |\mu_\ep|^2_{\gep})
\, + \,\chi_2 \, (|\mu_2|^2_{g_2}  - |\mu_\ep|^2_{\gep})
\, \big | & \leq & C_2 \,  \ep^{n-2} \, \rho_\ep^{1-n} \,,
\end{eqnarray}
for some positive constant $C_2>0$ independent of $\ep \in (0,1)$. To see this fact, we focus on the first summand on the left hand side (since the reasoning can be repeated for the second one) and we observe that it is dominated by $|\DI{\gep} X|_{\gep}^2 + |\mu -
\mu_1|_{\gep}^2$, where $X$ solves $L_{\gep} X = \sharp \, \dv{\gep} \mu$. Since $|\mu - \mu_1|_{\gep}^2$ is zero outside the
boundary of the poly-neck and since it is clearly bounded in the middle, we can concentrate on the squared pointwise norm of $D_{\gep} X$.

To keep this term under control, we recall from Proposition \ref{uniform estimate} that
there exists a positive constant $C_0>0$ such that for sufficiently small $\ep$'s
\begin{eqnarray*}
|D_{\gep} X|_{\gep} & \leq & C_0 \, \,
\nor{\dv{\gep}\mu}{{E}^1_{\delta-2}} \, \,
\rho_\ep^{\delta - 1} \, .
\end{eqnarray*}
At this point it is clear that if we show that $|D_{\gep}X|_{\gep}^2 \, \leq \, C_3 \,
\ep^{n-2} \, \rho_\ep^{1-n} $, for some $C_3>0$, then our claim \eqref{mu bound} will follows at once. Thus, thanks to the last inequality, it is sufficient to show that
\begin{eqnarray*}
\nor{\dv{\gep}\mu}{E^1_{\delta-2}} & \leq & C_4
\,\, \ep^{\frac{n-2}{2}} \, \rho_\ep^{\frac{3-n}{2}-\delta} \, ,
\end{eqnarray*}
for some positive constant $C_4>0$. Since $2-n  <  \delta <
 (2-n)/2$, this reduces to proving that
\begin{eqnarray}
\quad |\dv{\gep} \mu|_{\gep} \, \leq \, C_5 \,\, \ep^{\frac{1}{2} - \delta}  \rho_\ep^{\delta - 2} \,\, & \hbox{and } & \,\,   |\nabla \, \dv{\gep} \mu_\ep |_{\gep}  \, \leq \, C_6 \,\,  \ep^{\frac{1}{2} - \delta}  \rho_\ep^{\delta - 3} \,.
\end{eqnarray}
for some positive constants $C_5, C_6 > 0$. If we choose the
constant $0\, < \, c \, < \, 1$ which appears in the construction of
$\mu$ in Section \ref{geom construction} to be $c \, = \, 3/(4 - 2 \delta) $, then both
these conditions are satisfied and we get the desired bound for $|
\DI{\gep} X |^2_{\gep}$. The claim \eqref{mu bound} is now proven and we can deduce that the pure error term can be estimated as
$$
|F_{\ep}^{(0)}| \,\, \leq \, C_7 \, \ep^{n-2} \, \rho_\ep^{1-n} \,\, .
$$
As a consequence, we immediately get the following estimate
$$
\rho_\ep^{2-\gamma} \,|F_\ep(v)| \,\, \leq \,\, C_8 \, \big( \, 1 + |v| + |v|^2 \, \big) \, \ep^{1-\gamma} \, ( \cosh t )^{3-n-\gamma} + \, C_9 \, |v|^2 \, \ep^{2-\gamma} (\cosh t)^{2-\gamma} \,.
$$
Now suppose that $v$ lies in a ball of small radius $r_\ep= \ep^\mu$ in $C^0(M_\ep)$, with $\mu>0$ to be determined, and combine the estimate above with \eqref{bound} to obtain
$$
\nor{\mathcal{L}_{\gep}^{-1} \circ F_{\ep}(v)}{F^0_\gamma} \,\, \leq \,\, C_{10} \, \big( \, 1 + r_\ep + r_\ep^2 \, \big) \, \big[\, \ep^{1-\gamma} + \ep^{n-2}  \, \big] \,\, + \,\, C_{11} \, r_\ep^2 \, .
$$
The definition of the $\gamma$-weighted norm implies
$$
|\mathcal{L}_{\gep}^{-1} \circ F_{\ep}(v)| \,\, \leq \,\, C_{12} \big( \, 1 + r_\ep + r_\ep^2 \, \big) \, \big[\, \ep + \ep^{n-2+\gamma}  \, \big] \,\, + \,\, C_{13} \, r_\ep^2 \, .
$$
Now, it is sufficient to choose $a$ such that $0 < a < \min \{1, (n-2) +\gamma\}$ to conclude that for $\ep$ sufficiently small the mapping $\mathcal{L}_{\gep}^{-1} \circ F_{\ep}$ send the ball of radius $r_\ep = \ep^a$ centered at the origin of $C^0(M_\ep)$ in itself. The fact that $\mathcal{L}_{\gep}^{-1} \circ F_{\ep}$ is also a contraction in this ball follows from a direct and easier computation, whose details are left to the reader.

Provided $\ep$ is chosen in a sufficiently small range, we can apply the contraction mapping Theorem and we obtain a fixed point $v_\ep$ for \eqref{fixed} and thus a solution $\tilde u_\ep =1+v_\ep$ to \eqref{licn}. The fact that $v_\ep$ and hence $\tilde u_\ep$ is smooth, follows from the elliptic regularity theory, using a standard boot strap argument. Since $v_\ep$ lies in a ball of radius $r_\ep = \ep^{a}$ in $C^0(M_\ep)$, we deduce at once that the new solutions of the constraint equations converge to hold ones on the compact sets of $M_i\setminus K$, $i=1,2$, with respect to the $C^0$-topology, as $\ep \rightarrow 0$.

To obtain the smooth convergence claimed in the statement of Theorem \ref{teo 1}, it is now sufficient to combine this $C^0$-bound with interior Schauder estimate. To see this fact we show how it is possible to obtain the $C^{2,\al}$-control on the compact subsets of $M_1 \setminus K$. The same argument can be applied with minor changes in order to obtain $C^{k,\al}$-convergence on the compact subsets of $M_i \setminus K$, for $i=1,2$ and every $k \in \mathbb{N}$. We fix now a compact subset $Q \subset M_1 \setminus K$ and thanks to \cite[Proposition 6.2]{Gil-Tru}, we have that
\begin{eqnarray}
\label{int schauder}
\nor{v_\ep}{C^{2,\al}(Q)}^{(0)} & \leq & A_0 \,\, \big[ \, \nor{F_{\ep} (v_{\ep})}{C^{0,\al}(Q)}^{(2)}  + \, \nor{v_\ep}{C^{0}(Q)}     \,\big] \,,
\end{eqnarray}
where the interior norms $\nor{\cdot}{C^{k,\al}(Q)}^{(\sigma)}$ correspond to the interior norms $|\cdot|_{k,\al;Q}^{(\sigma)}$ defined in \cite{Gil-Tru}. From \eqref{int schauder} we deduce that there exists a positive constant $A_1>0$ such that
\begin{eqnarray*}
\nor{v_\ep}{C^{2,\al}(Q)}^{(0)} & \leq & A_1 \,\, \big[\,\, \nor{F_\ep^{(0)}}{C^{1}(Q)}^{(2)}  + \, \nor{v_\ep}{C^{0}(Q)} + \, \sup_Q d_p^2 \, |v_\ep|^2 \\
& & +  \sup_Q \, d^2_p \, |v_\ep| \, |R_{g_1}-R_{\gep}| \, + \sup_Q \, d_p^3 \, |v_\ep| \, |\nabla (R_{g_1} - R_{\gep})| \\
& & +  \sup_Q \, d^2_p \, |v_\ep| \, \big||\mu_1|^2_{g_1}-|\mu_\ep|^2_{\gep}\big| \, + \sup_Q \, d_p^3 \, |v_\ep| \, |\nabla (|\mu_1|^2_{g_1}-|\mu_\ep|^2_{\gep})| \\
& &  +  \sup_Q \, d^3_p \, |\nabla v_\ep| \, |R_{g_1}-R_{\gep}| \, + \sup_Q \, d^3_p \, |\nabla v_\ep| \, \big||\mu_1|^2_{g_1}-|\mu_\ep|^2_{\gep}\big|
\\
& & + \sup_Q \, d_p^3 \, |\nabla v_\ep| \, |v_\ep| \,\, \big] \,,
\end{eqnarray*}
where $d_p$ is a short notation for the distance to the boundary of the compact set $Q$, namely $d_p := dist_{g_1}(p, \partial Q)$. Now it is easy to check that, except for the last three summands, the right hand side tends to zero, as $\ep \rightarrow 0$. On the other hand, for $\ep$ sufficiently small, these last three terms involving the first derivative of $v_\ep$ can be absorbed by the left hand side. Hence, there exists a positive constant $A_2>0$ and there exists a real number $\ep_0$, such that for every $\ep \in (0, \ep_0]$
\begin{eqnarray*}
\nor{v_\ep}{C^{2,\al}(Q)}^{(0)} & \leq & A_2 \,\, \big[\,\, \nor{F_\ep^{(0)}}{C^{1}(Q)}^{(2)}  + \, \nor{v_\ep}{C^{0}(Q)}
+ \, \nor{v_\ep}{C^{0}(Q)}^{(2)} \, \nor{v_\ep}{C^{0}(Q)} \\
& & +\, \nor{ R_{g_1} - R_{g_\ep} }{C^{0}(Q)}^{(2)} \nor{v_\ep}{C^{0}(Q)} +\, \nor{ |\mu_1|^2_{g_1}-|\mu_\ep|^2_{\gep}}{C^{0}(Q)}^{(2)} \nor{v_\ep}{C^{0}(Q)} \\
& & +\, \nor{ \nabla (R_{g_1} - R_{g_\ep}) }{C^{0}(Q)}^{(3)} \nor{v_\ep}{C^{0}(Q)} \\
& &  +\, \nor{ \nabla ( |\mu_1|^2_{g_1}-|\mu_\ep|^2_{\gep} ) }{C^{0}(Q)}^{(3)} \nor{v_\ep}{C^{0}(Q)} \,\, \big] \,.
\end{eqnarray*}
We deduce that $v_\ep \rightarrow 0$ on the compact subset of $M_1 \setminus K$ with respect to the $C^{2,\al}$-topology. As we have already observed, this argument can be easily modified in order to deduce the smooth convergence on the compact subsets of $M_i\setminus K$, $i=1,2$. This concludes the proof of Theorem \ref{teo 1}.

\section{Localized gluing}\label{secLoc}

In this  section, we give the proof of Theorem \ref{Tloc}.
We will then show that the generalized conformal gluing procedure used in \cite{Mazz 3}
for compact manifolds and here for AE and AH manifolds, can be, in a second
hand, localized  in order to keep the original data outside a small neighborhood
of the neck. This technique, inspired by the Corvino Shoen method \cite{CS}, was already employed in \cite{CD2} to localize the gluing
of Isenberg, Mazzeo and Pollack \cite{IMP}.

\proof Let $\chi$ be a non negative smooth   cut-off function
equal to $1$ in a neighborhood of $\partial K^i_{\ep_1}$ and equal to zero in a
neighborhood of $\partial K^i_{2\ep_1}$. On $\Gamma^i(\ep_1,2\ep_1)$ set
$$\mathring{\Pi}_\ep = \chi \tilde \Pi_\ep + (1-\chi) \Pi_i\;,$$
$$\mathring{g}_\ep = \chi \tilde  g_\ep + (1-\chi) g_i\;.$$
Then $(\mathring{g}_\ep,\mathring{\Pi}_\ep)$ coincides with the data
$( \tilde g_\ep,\tilde \Pi_\ep)$ in a neighborhood of $\partial K^i_{\ep_1}$, and
coincides with the original data $(g_i,\Pi_i)$ in a neighborhood of
$\partial K^i_{2\ep_1}$. It follows that $\rho
(\mathring{g}_t,\mathring{\Pi}_t)$ and $
J(\mathring{g}_t,\mathring{\Pi}_t)$ are supported away from the
boundary in $\Gamma^i(\ep_1,2\ep_1)$. Since the data constructed in \cite{Mazz 3} or here in the AE and AH case, converge
uniformly, in any $C^{k, \alpha}$ norm, to the original ones on $\Gamma^i(\ep_1,2\ep_1)$ we will
have
$$\lim_{\ep\to 0} \rho (\mathring{K}_\ep,\mathring{g}_\ep) =0=\lim_{\ep\to 0}
J(\mathring{K}_\ep,\mathring{g}_\ep)\;.$$  Theorem 5.9 and
Corollary 5.11 of \cite{CD2} provides  $0<\ep_2\leq \ep_1$ such that for all
$0<\ep\leq \ep_2$ there exists a solution
$(\hat{g}_\ep,\hat{\Pi}_\ep)$ of the vacuum constraint
equations which is smoothly extended by $(\tilde g_\ep,\tilde {\Pi}_\ep)$
across $\partial K^i_{\ep_1}$ and by $(g_i,\Pi_i)$ across $\partial K^i_{2\ep_1}$.
\qed

\section{Constraint with cosmological constant and constant scalar curvature metrics}

The previous constructions (conformal and local gluing) can be adapted to the context of vacuum Einstein constraint equations with cosmological constant
\begin{eqnarray}
\label{ECE 1 Lambda}
J(g,\Pi):=\dv{g} \, \Pi - \di \left( \tr{g} \, \Pi \right) & = & 0  \\
\label{ECE 2 Lambda} \rho(g,\Pi):=R_g - |\Pi|_{g}^2 + \left( \tr{g} \, \Pi \right)^2 & =
& 2\Lambda \, .
\end{eqnarray}
As an interesting example we mention the case of the hyperbolic space $\mathbb{H}^m$ viewed as horizontal slice in AdS space time. In fact it is well known that this object satisfies the equations above with $\Lambda=-m(m-1)/2$ and $\Pi=0$. To adapt our construction to this case it is sufficient to replace the condition $\tau=m$ (which is natural when $\mathbb{H}^m$ is viewed as a slice in Minkowsky space-time) by $\tau=0$, as in the AE case.

Others interesting examples are given by constant scalar curvature Riemannian metrics. In fact these are time symmetric ($\Pi=0$) solutions
to the vacuum Einstein contraint with cosmological constant $\Lambda$.
In \cite{Mazz} and \cite{Mazz 2} generalized connected sum of such metrics have been obtained by the second author.
The adaptation of these conformal gluing to the non compact AE and AH context proceed exactly as in our proof of Theorem \ref{teo 1} and in fact it is easier.
Also, the adaptation of the localized gluing proceed exactly as in Section \ref{secLoc}, the no KID's condition
being here that the adjoint $DR(g)^*$ of the linearized scalar curvature operator has no kernel on $\Gamma^i(\ep_1,2\ep_1)$.

\section{Concluding remarks}

We have shown that the generalized connected sum of some class of solutions to the Einstein constraints equations
can also be performed in natural non compact setting such as AE or AH manifolds, as far as the submanifold we cut around is compact. In addition, we have shown that whatever the manifold is compact or not, the conformal gluing procedure can be localized in a neighborhood of the poly-neck region.

To conclude, we list some open questions which will be the object of further investigation.

1) In the non compact setting it would be natural to consider the connected sum along non compact submanifold so :

a) It seems that the conformal gluing could be adapted to this situation, provided the submanifold itself present a suitable asymptotic behavior.

b) The `localized gluing' need also to be understood in this case. A closely related question is the gluing along an horizontal band in the (half space type) AH context. It seems to us that in this situation one could overcome some of the scaling difficulties encountered in \cite{CD5}.

2) So far, the localized gluing procedure can be applied only after the conformal one is performed. It will
be interesting to understand under which hypothesis one can produce a localized gluing directly (see the unfortunate tentative in \cite{CD1} section 5 withdraw  in  \cite{CD1bis}).

3) It will be interesting to see if the generalized gluing  can produce black hole with special topology.\\
\\
\begin{small}
\noindent \textbf{Acknowledgements.}
{\em The first author is
grateful to the french ANR project `GeomEinstein' for financial support. The second author is supported by the Italian project FIRB--IDEAS `Analysis and Beyond' and he is grateful to the University of Avignon for the hospitality during the preparation of this work.}
\end{small}

\end{document}